\documentclass[11pt]{amsart}
\usepackage{epic} 
\usepackage{amscd,amsfonts,amsmath,amsxtra,amssymb,epsfig,color}
\usepackage[all]{xy}
\usepackage{latexsym,a4}
\usepackage{amssymb}
\usepackage{hyperref}

\numberwithin{equation}{section}

\begin{document}

\newcommand{\filename}{four-cusps-2017-05-18.tex}

\newcommand{\PP}{\mathbb P}
\newcommand{\ZZ}{\mathbb Z}
\newcommand{\lra}{\longrightarrow}
\newcommand{\QQ}{\mathbb Q}
\newcommand{\CC}{{\mathbb C}}
\def\OO{{\mathcal O}}
\newcommand{\Jac}{\operatorname{Jac}}
\newcommand{\WW}{S}
\newcommand{\RES}{X}

\newcommand{\Xt}{\RES_{3,3,3,3}}
\newcommand{\Ft}{\mathcal F_{3,3,3,3}}
\newcommand{\Xf}{\RES_{4,3,1}}
\newcommand{\Ff}{\mathcal F_{4,3,1}}
\newcommand{\Yt}{\XXX_{3,3,3,3}}
\newcommand{\Yf}{\XXX_{4,3,1}}

\newcommand{\Xs}{\RES_{6,3,2,1}}
\newcommand{\Ys}{\XXX_{6,3,2,1}}
\newcommand{\Ss}{\WW_{6,3,2,1}}

\newcommand{\FF}{E}
\newcommand{\FFb}{F}
\newcommand{\FFc}{C}
\newcommand{\FFF}{{\mathbb F}}
\newcommand{\lpsi}{\mathfrak q}
\newcommand{\homeo}{\alpha}
\newcommand{\pbi}{\beta}
\newcommand{\elpi}{\varphi}
\newcommand{\proj}{\pi}
\newcommand{\XXX}{Y}
\newcommand{\deck}{\psi}

\newcommand{\HE}{H}

\newcommand{\PLpsi}{{\mathfrak p}}

\newcommand{\Ee}{e}

\newcommand{\mX}{\mathcal Y'}
\newcommand{\mY}{\mathcal Y}
\newcommand{\mf}{\mathfrak p}

\newcommand{\Num}{\operatorname{Num}}
\newcommand{\Pic}{\operatorname{Pic}}
\newcommand{\NS}{\operatorname{NS}}
\newcommand{\MW}{\operatorname{MW}}
\newcommand{\supp}{\operatorname{supp}}

\title{On Enriques surfaces with four cusps}
\author{S\l awomir Rams}
\address{Institute of Mathematics, Jagiellonian University, 
ul. {\L}ojasiewicza 6,  30-348 Krak\'ow, Poland}
\email{slawomir.rams@uj.edu.pl}

\author{Matthias Sch\"utt}
\address{Institut f\"ur Algebraische Geometrie, Leibniz Universit\"at
  Hannover, Welfengarten 1, 30167 Hannover, Germany}
    \address{Riemann Center for Geometry and Physics, Leibniz Universit\"at
  Hannover, Appelstrasse 2, 30167 Hannover, Germany}

\email{schuett@math.uni-hannover.de}

%
%

\thanks{Funding   by ERC StG~279723 (SURFARI) and  NCN grant N N201 608040 
 is gratefully acknowledged.}


\date{May 22, 2017}

\begin{abstract}
 We study
Enriques surfaces with four disjoint A$_2$-configurations. In particular, we construct open Enriques surfaces with 
fundamental groups $(\ZZ/3\ZZ)^{\oplus 2} \times \ZZ/2\ZZ$ and  $\ZZ/6\ZZ$, 
completing the picture of the A$_2$-case from \cite{kz}.
We also construct an explicit Gorenstein $\QQ$-homology projective plane
of singularity type A$_3+3$A$_2$,
supporting an open case from \cite{HKO}.
\end{abstract}

\maketitle


\newcommand{\D}{\displaystyle}
\newcommand{\vv}{\vspace{0.3cm}}
\newcommand{\vV}{\vspace{0.5cm}}
\newcommand{\VV}{\vspace{2cm}}

\theoremstyle{remark}
\newtheorem{obs}{Observation}[section]
\newtheorem{rem}[obs]{Remark}
\newtheorem{examp}[obs]{Example}
\newtheorem{notation}[obs]{Notation}
\newtheorem{Claim}[obs]{Claim}
\newtheorem{Tool}[obs]{Tool}

\theoremstyle{definition}
\newtheorem{defi}[obs]{Definition}
\newtheorem{prop}[obs]{Proposition}
\newtheorem{thm}[obs]{Theorem}
\newtheorem{lemm}[obs]{Lemma}
\newtheorem{cor}[obs]{Corollary}
\newtheorem*{THM}{Proposition}
\newcommand{\ux}{\underline{x}}


\section{Introduction}


The main aim of this note is to study 
 Enriques surfaces with four disjoint A$_2$-configurations,
 the maximum number possible
 (because an Enriques surface
 has Picard number $10$). 
We 
shall make heavy use of elliptic fibrations to study the moduli of  such Enriques surfaces:

\begin{thm}
\label{thm0}
Enriques surfaces  with  four disjoint A$_2$-configurations 
come in exactly two irreducible two-dimensional families $\Ft, \Ff$.
\end{thm}
This result which relies on the understanding of
Picard-Lefschetz reflections on the Enriques surface and its K3-cover following \cite{kz}, 
enables us to determine the fundamental groups of the open Enriques surfaces 
obtained by removing the A$_2$-configurations (often also referred to as the cusps).

Our paper draws on the classification of possible  
fundamental groups of open
Enriques surfaces (i.e. complements of configurations of smooth rational curves)
initiated in \cite{kz}. 
Keum and Zhang state  a list of $26$ possible groups and give $24$ examples. 
Here we supplement and correct their results by adding one example and one group supported by another example.
Our second main result is as follows.

\begin{thm}
\label{thm1}
Let $G\in\{S_3 \times \ZZ/3\ZZ, (\ZZ/3\ZZ)^{\oplus 2} \times \ZZ/2\ZZ, \ZZ/6\ZZ\}$.
Then there is a complex Enriques surface $S$
with a set $\mathcal A$ of four disjoint A$_2$-configurations 
such that
\[
\pi_1(S\setminus \mathcal A)\cong G.
\]
\end{thm}

For a more concise statement, the reader is referred to  Theorem \ref{thm}.
This completes the picture of the A$_2$-case.


Another key point of our paper is
the clarification that there are indeed Enriques surfaces admitting different sets of 
four disjoint A$_2$-configurations which lead to each alternative of the fundamental group in Theorem \ref{thm1}.
This issue will be discussed in detail in Section \ref{s:moduli}
and also supported by an explicit example, see \S\ref{ss:explicit}.

While some of the constructions involved in our methods are analytic in nature,
notably the notion of logarithmic transformations of elliptic surfaces,
we will crucially facilitate Enriques involutions of base change type as studied systematically in \cite{hm1}
since this algebro-geometric technique grants us good control of the curves on the surfaces and their moduli.
We review this among all the prerequisites and basics necessary for the understanding of this paper 
in Section \ref{s-lat}.
Section \ref{s-2fam} introduces the two families $\Ft$ and $\Ff$
and proves Theorem \ref{thm0}.
The proof of Theorem \ref{thm1}
 is given in Section \ref{s-fundamental groups}.
As a biproduct, our examples produce an explicit Gorenstein $\QQ$-homology projective plane
of singularity type A$_3+3$A$_2$;
by Proposition \ref{prop} we settle an open case from \cite{HKO}.
The paper concludes with further considerations concerning the moduli of Enriques surfaces 
with 4 disjoint A$_2$-configurations.

\vspace{.3cm}

\noindent
{\it Convention:} 
In this note  the base field is always $\CC$.
Root lattices $A_n, D_k, E_l$ are taken to be negative definite.

\section{Preliminaries and basics}
\label{s-lat}


\subsection{A$_2$-configurations}
\label{ss:A_2}

Let $\WW$ be an Enriques surface that   
contains four disjoint A$_2$-configurations, i.e.~eight  smooth rational curves $\FFb_1', \FFb_1'', \ldots, \FFb_4', \FFb_4''$ such that 
$$
\FFb_{j}'.\FFb_{j}'' = 1 \mbox{ for } j= 1,\ldots, 4 \, ,
$$
and the rational curves in question are mutually disjoint otherwise. 
We say that a collection of disjoint A$_2$-configurations $\FFb'_1,  \FFb''_1, \ldots,\FFb'_l, \FFb''_l$ 
is {\sl three-divisible}
if 
one can label the rational curves in each A$_2$-configuration such that the divisor 
\begin{equation} \label{eq-3div}
\sum_{j=1}^{l} (\FFb'_{j} - \FFb''_{j})      
\end{equation}
is divisible by $3$ in $\mbox{Pic}(\WW)$. 
Equivalently,
since 
$$
\mbox{Pic}(\WW) = \mbox{H}^2(\WW,\ZZ) = \ZZ^{10} \oplus \ZZ/2\ZZ,
$$
the class given by \eqref{eq-3div}  is $3$-divisible in $\mbox{Num}(\WW)$,
the lattice given by divisors up to numerical equivalence. 
Recall that $\Num(S)$ is a unimodular even hyperbolic lattice; in fact
\[
\Num(\WW) = U + E_8
\]
where $U$ denotes the unimodular hyperbolic plane. 
The following result (which is similar to \cite[Lem.~1]{b} and \cite[Lem.~1.1]{br}) 
will be instrumental for some of our investigations.

\begin{lemm}
\label{lem:l=3}
Let $\sum_{j=1}^{l} (\FFb'_{j} - \FFb''_{j})=3D$ in $\Num(S)$. Then $l=3$, but $D$ is neither effective nor anti-effective.
\end{lemm}

\begin{proof}
The fact $l=3$ is  easily derived
using
$(\FFb'_j-\FFb''_j)^2=-6$ and  the integrality and rank of the lattice $\Num(\WW)$.
Assume that $D$ is effective.
Then since $F_j'.D=-1$ for each $j=1,2,3$, each of these curves $F_j'$ is contained in the support of $D$.
Hence $D'=D-(F_1'+F_2'+F_3')$ is still effective.
On the other hand, we obviously have $3D'\leq 0$,
hence $3D\sim 0$, but this is not compatible with $D^2=(D')^2=-2$,
contradiction.

If $D\leq 0$, then an analogous argument applies to the $F_j''$,
thus completing the proof of the lemma.
%
\end{proof}


 We follow the approach of  \cite[$\S$~3]{kz} and let  $M$ (resp. $\overline{M}$) 
 denote the lattice spanned by $\FFb'_1, \ldots, \FFb''_4$
 in $\mbox{Num}(\WW)$ (resp. its primitive closure). 
 
 \begin{lemm}
 \label{lem:M}
 The index of $M$ inside $\overline M$ satisfies $[\overline M:M]\in\{3,3^2\}$.
 \end{lemm}
 
 \begin{proof}
The lattice $M$ has  discriminant
 $d(M) = 3^4$, so $[\overline{M}:M] \in \{1, 3, 3^2\}$.
We claim that the first case is impossible. 
Indeed, suppose that $M = \overline{M}$. 
Then $M\hookrightarrow\Num(\WW)$ is a primitive embedding,
so
\[
M^\vee/M \cong (M^\perp)^\vee/M^\perp.
\]
By assumption
the left-hand side is isomorphic to $(\ZZ/3\ZZ)^4$ while the right-hand side 
comes from the rank-$2$ lattice $M^\perp$,
thus has  length at most 2,
contradiction. 
\end{proof}

\begin{cor} \label{eq-2poss} 
The four A$_2$-configurations
$\FFb'_1, \ldots, \FFb''_4$
contain  either one or four 
$3$-divisible sets.
\end{cor}

In particular, we can infer that
\begin{equation} \label{eq-unimod}
\FFb'_1, \ldots, \FFb''_4 \mbox{ contain four } 3\mbox{-divisible sets if and only if } \; \overline{M} \mbox{ is unimodular.}
\end{equation}
In other words, in this case each triplet of the A$_2$-configurations in question is $3$-divisible
up to relabelling the rational curves.

\subsection{Elliptic fibrations}
\label{ss:ell}

We start by recalling some basic concepts and relations.
Any complex Enriques surface $\WW$ admits an elliptic fibration
\begin{eqnarray}
\label{eq:pi}
\varphi:\;\; \WW \to \PP^1.
\end{eqnarray}
There are two fibres of multiplicity two;
their supports are usually called half-pencils.
The difference of the two half-pencils gives the canonical divisor
which represents the two-torsion in $\mbox{H}^2(\WW,\ZZ)$.
This already shows that the fibration cannot have a section,
but by \cite[Prop. VIII.17.6]{bpv} there always is a bisection $R$ of square $R^2=0$ or $-2$,
i.e. an irreducible curve $R$ such that $R.F=2$ for any fiber $F$ of \eqref{eq:pi}.

The moduli of Enriques surfaces can be studied through the universal cover 
\begin{eqnarray}
\label{eq:K3}
\pi: \XXX\to\WW
\end{eqnarray}
which is a K3 surface.
By construction, this induces an elliptic fibration 
\begin{eqnarray}
\label{eq:K3-ell}
\tilde\varphi: \;\;\; \XXX\to\PP^1
\end{eqnarray} 
which fits into the commutative diagram
\begin{equation} \label{firstdiagram}
\begin{xy}
\xymatrix{\XXX   \ar[r]^{2:1} \ar[d]_{\tilde{\elpi}} & \WW \ar[d]_{\elpi} \\
\PP^1  \ar[r]^{2:1}               & \PP^1 
}
\end{xy}
\end{equation}
The bottom row degree-2 morphism
\begin{eqnarray}
\label{eq:base-change}
\PP^1\stackrel{2:1}{\longrightarrow}\PP^1
\end{eqnarray}
ramifies exactly in the points below the multiple fibres.
Moreover, the universal covering induces 
a  primitive embedding
\[
U(2) + E_8(2) \cong\pi^*\Num(\WW) \hookrightarrow \Pic(\XXX)
\]
which lends itself to a study of K3 surfaces with the above lattice polarisation.
Abstractly,
a complex K3 surface $Y$ admits an Enriques involution
if and only if there is a primitive embedding of $U(2) + E_8(2)$ into $\Pic(Y)$ without perpendicular roots
(i.e.~classes of smooth rational curves) by \cite{Keum-Kummer}.
In view of this, it is evident that a bisection $R$ of square $R^2=0$ occurs generically,
since on the contrary any  $(-2)$-curve on $\WW$ necessarily splits into two disjoint 
smooth rational curves on the K3 cover $\XXX$;
these give sections of \eqref{eq:K3-ell}, causing the Picard number to go up to 11 at least.
The same generic behaviour will occur on our families $\Ft, \Ff$ in Section \ref{s-2fam}.

On the other hand, we can consider the Jacobian fibration of \eqref{eq:pi}.
This will be a rational elliptic surface 
\begin{eqnarray}
\label{eq:RES}
\RES\to \PP^1
\end{eqnarray} 
with section
and  is thus governed by means of explicit classifications, e.g. using the theory of Mordell-Weil lattices in \cite{OS}.
Naturally $\WW$ and $\RES$ share the same singular fibers,
except that on $\WW$,
 smooth or semi-stable fibers (Kodaira type $I_n, n\geq 0$) may come with multiplicity two.
The Enriques surface $\WW$ can be recovered from $\RES$ through a logarithmic transformation
which depends on the choice of non-trivial 2-torsion points 
in two distinct smooth or semi-stable fibers of \eqref{eq:RES}  (see e.g. \cite[\S~1.6]{fm}).
Intrinsically this leads to another K3 surface
in terms of the jacobian elliptic fibration
arising from \eqref{eq:RES} through the quadratic base change \eqref{eq:base-change}
ramified in the two distinct fibers 
where the logarithmic transformation changed the multiplicities of fibers.
It is clear from the construction,
that at the same time this K3 surface features as the Jacobian of \eqref{eq:K3-ell}.
That is, we get another commutative diagram
\begin{equation}
\begin{xy} \label{seconddiagram}
\xymatrix{\Jac(\XXX)   \ar[r]^{2:1} \ar[d] & X \ar[d] \\
\PP^1  \ar[r]^{2:1}               & \PP^1   }
\end{xy}
\end{equation}
Recall that  the depicted elliptic fibrations on $\XXX$ and  $\Jac(\XXX)$ 
share the same configurations of
singular fibers 
and the same Picard numbers. 

For some purposes, the above construction has the drawback
of being analytical in nature.
This can be circumvented in the special situation
where the elliptic fibration \eqref{eq:K3-ell} is already jacobian,
i.e. admits a section.
For instance, this occurs in the presence of a bisection $R$ of \eqref{eq:pi} with square $R^2=-2$
as indicated above.
A more general framework for this to occur was introduced 
in terms of involutions of base change type in \cite{hm1}.
Here one considers the quadratic twist $\RES'$ of $\RES$
which acquires $I_n^*$ fibres ($n\geq 0$) at the two ramification points of \eqref{eq:base-change}.
In consequence, the quadratic base change \eqref{eq:base-change} applied to either $\RES$ and $\RES'$
gives the same K3 surface $\XXX$.

For any section on $\RES'$, the pull-back to $\XXX$ is anti-invariant with respect to the involution $\imath$ on $\XXX$
induced by the deck transformation of \eqref{eq:base-change} (s.t. $\XXX/\imath = \RES$).
It follows that $\imath$ composed with translation by the anti-invariant section
defines another non-symplectic involution on $\XXX$.
This has fixed points, necessarily in the ramified fibers, if and only if the section meets the identity components of the 
two twisted fibres on $\RES'$.
Otherwise, for instance if the section is two-torsion,
we obtain an Enriques involution on $\XXX$
which we will refer to as an involution of base change type.

\subsection{Picard-Lefschetz reflections}

Recall that by Kodaira's work \cite{K},
the irreducible components of a singular fibre of an elliptic fibration correspond to an extended Dynkin diagram;
a Dynkin diagram, or equivalently root lattice, can be obtained 
from the singular fiber by omitting any simple component. 
Given A$_2$-configurations,
it is thus natural to ask whether these correspond to rational curves
supported on the fibres of an elliptic fibration on $\WW$.
While this may not be true in general,
we can weaken the limitations
by considering the question up to isometries of $\mbox{H}^2(\WW,\ZZ)$.
This will allow us to reduce the problem of 3-divisible sets of A$_2$-configurations
to the study of certain elliptic fibrations on Enriques surfaces.
To this end, 
recall that 
each $(-2)$-class $E$ in $\mbox{H}^2(\WW,\ZZ)$
defines a {\sl Picard-Lefschetz reflection} 
$$
\mbox{s}_{\FF} \, : \,  \mbox{H}^2(\WW,\ZZ) \ni D \mapsto D + (D.\FF) \FF \in \mbox{H}^2(\WW,\ZZ).
$$
In general, such a reflection does not act effectively on divisors,
but the situation changes drastically when restricted to smooth rational curves.
Namely, if $E$ and $E'$ are both represented by a smooth rational curve,
then
\begin{eqnarray}
\label{eq:s(E)}
\;\;\;\;\;\;\;\;\;
\mbox{s}_E(E') \text{ is \; either effective (if } E\neq E' \text{) \; or anti-effective (if } E=E').
\end{eqnarray}
In the sequel we will use the following corrected version of  
\cite[Claim~3.5.1]{kz} (which included neither  the configurations \eqref{eq-ppss}
nor the degenerate case of \eqref{eq-nppss}, cf. also Remark \ref{rem:KZ}). 

\begin{lemm} \label{lem-1-2} 
There exists a half-pencil $\HE$ on $\WW$
and smooth rational curves $E_1, \ldots , E_k \subset \WW$
such that the image of  each curve $\FFb'_j$, $\FFb_j''$, where $j = 1, \ldots,4$, 
under the isometry
\begin{equation}  \label{eq-plenr}
 \PLpsi_{\WW} := (s_{E_k} \circ    \ldots \circ s_{E_1})
\end{equation}
is, up to some multiple of $\HE$, the class of a smooth rational curve which is an irreducible component of a member  of the pencil $|2 \HE|$.
Moreover, the elliptic fibration given by  $|2 \HE|$ is either of the type 
\begin{equation} \label{eq-ppss}
I_3^4, I_3^3 \oplus 2 I_3,   I_3^2 \oplus (2 I_3)^2
\end{equation}
 or of the type 
\begin{equation} \label{eq-nppss}
IV^{*} \oplus I_3 \oplus I_1,   IV^{*} \oplus 2I_3 \oplus I_1,    IV^{*} \oplus I_3 \oplus 2I_1, IV^{*} \oplus 2I_3 \oplus 2I_1, IV^*\oplus IV.
\end{equation}
\end{lemm}

\begin{proof}
We argue with the lattices $M, M^\perp\subset\Num(S)$ from \ref{ss:A_2}.
Since $M$ has discriminant $81$ and $M^\perp$ is hyperbolic of rank $2$, $M^\perp$ represents zero.
Thus there is an isotropic class $\HE_0\in M^\perp$
which we may assume to be primitive in H$^2(S,\ZZ)$.
By Riemann-Roch,
either $\HE_0$ or $-\HE_0$ is effective,
so let us assume the former.
Following \cite[Lemma~VIII.17.4]{bpv},
it remains to subtract the base locus of $|2\HE_0|$
to derive an elliptic fibration.
This precisely amounts to a composition 
\begin{equation} \label{eq-defppzero} 
\mathfrak p_0 = (s_{E_l} \circ    \ldots \circ s_{E_1})
\end{equation} 
of reflections in  smooth rational curves $E_1, \ldots, E_l$
(each meeting the image of $\HE_0$ under the previous reflections negatively).
By construction, we obtain the half-pencil $\HE:=\mathfrak p_0(\HE_0)$
such that $|2 \HE|$ induces an elliptic fibration on $S$.

To study the impact of the isometry $\mathfrak p_0$ on the $(-2)$-classes $F_j', F_j''$,
the following generalization of \eqref{eq:s(E)} enters crucially:

\begin{Claim}
\label{claim}
For $j=1,\hdots,4$
the class $\mathfrak p_0(F_j')$ (resp. $\mathfrak p_0(F_j'')$)  is an effective or anti-effective divisor
supported on components of a singular fiber of  the elliptic pencil $|2 \HE|$. 
\end{Claim}

\noindent
In order to simplify the exposition of the proof of the lemma, the proof of Claim~\ref{claim} will be given, also for later use, in \ref{ss:PL-K3}.

Continuing the proof of Lemma \ref{lem-1-2}, we infer  from Claim \ref{claim} that singular fibers of the elliptic pencil $|2 \HE|$
contain  4 disjoint A$_2$-configurations
(given by effective or anti-effective $(-2)$-divisors,
but not necessarily (yet) by irreducible curves).
As explained in \ref{ss:ell},
the jacobian fibration of
$|2 \HE|$ is a rational elliptic surface $\RES$. 
As it shares the 4 disjoint A$_2$-divisor configurations in the fibers,
$\RES$ is automatically extremal by the Shioda-Tate formula \cite[Cor. 6.13]{shioda-schuett}, 
i.e. $\RES$ has finite Mordell-Weil group.
Going through the classification in \cite{MP},
one finds that $\RES$ may have the following configurations:
\begin{eqnarray}
\label{eq:configs}
I_3^4, \; IV^{*} \oplus I_3 \oplus I_1, \; IV^*\oplus IV, \; II^*\oplus I_1^2,\; II^* \oplus II.
\end{eqnarray}
Note that the configurations \eqref{eq-ppss} and \eqref{eq-nppss} from Lemma \ref{lem-1-2} 
correspond to the first three entries in \eqref{eq:configs}, with fiber multiplicities.
In order to complete the proof of Lemma \ref{lem-1-2}, we shall now 
prove all claims for the first three configurations above
before ruling out the last two configurations  from \eqref{eq:configs}.

Before going into the details, recall that for any root lattice $R$,
any two roots are equivalent under reflections.
Naturally this extends to the extended Dynkin diagrams $\tilde R$:

\begin{Tool}
\label{tool:roots}
Any two roots in $\tilde R$ are equivalent under reflections.
\end{Tool}

(Note that this holds true even though $\tilde R$ contains infinitely many roots
-- but only finitely many modulo the primitive isotropic vector).

\medskip

Consider the  configurations \eqref{eq-ppss} and \eqref{eq-nppss} from Lemma \ref{lem-1-2}.
To complete the proof, we have to 
show that there are reflections
in fiber components of $|2 \HE|$
such that the composition of all reflections takes each curve $F_j', F_j''$
to a single smooth rational curve, up to 
a multiple of $\HE$.
To see this, fix $D=\mathfrak p_0(F_j')$ or $\mathfrak p_0(F_j'')$ for some $j=1,\hdots,4$.
We claim that there is an integer $n\in\ZZ$ and a divisor $\tilde{D}$
such that 
\begin{eqnarray}
\label{eq:DED'}
D = nH + \tilde{D} \;\; \text{ and } \;\; 0 < \tilde{D} < 2 \HE.
\end{eqnarray}
This can be seen without difficulty because fiber components generate a semi-negative-definite lattice.
Indeed, if there were some $n\in\ZZ$ such that
\[
D = n \HE + \tilde{D} - \hat{D}
\]
with $\tilde{D}, \hat{D}$ both effective and supported on distinct fiber components,
then by construction
\[
-2 = D^2 = \tilde{D}^2 - 2\tilde{D}.\hat{D} + \hat{D}^2.
\]
Since all entries on the RHS are non-positive even integers,
we deduce that either $\tilde{D}$ or $\hat{D}$ has square zero,
hence equals some fiber multiple.
Upon subtracting or adding the fiber class $2 \HE$,
we thus obtain the representation \eqref{eq:DED'} of $D$.
Note that in particular $\tilde{D}$ is supported on a single fiber
(and naturally $\tilde{D}$ and $n$ can be chosen such that $0<\tilde{D}<\HE$
if $\tilde{D}$ is supported on a multiple fiber), so we can now 
complete the proof of Lemma \ref{lem-1-2} fiber by fiber.
In particular, we only have to distinguish two cases.

If the fiber 
has Kodaira type $I_3$ or $IV$,
with components $\Theta_0, \Theta_1, \Theta_2$ meeting each other transversally,
then up to permutations of components,
the only possibilities for a configuration given by effective $(-2)$-divisors $\tilde{D}_1, \tilde{D}_2<2H$
(or $<H$ if the fiber has multiplicity $2$) 
such that  $\tilde{D}_1.\tilde{D}_2=1$
are easily determined as
\[
\tilde{D}_1=\Theta_1, \tilde{D}_2=\Theta_2 \;\; \text{ and } \;\; \tilde{D}_1=\Theta_0+\Theta_1, \tilde{D}_2=\Theta_0+\Theta_2.
\]
Since the second configuration is obtained from the first by reflection in $\Theta_0$,
the remaining statement of Lemma \ref{lem-1-2} holds on fibers of type $I_3$ and $IV$.

\medskip

Suppose the fiber in question has Kodaira type $IV^*$,
thus supporting 3 disjoint A$_2$-type configurations   given by the divisors $\tilde{D}$ 
obtained from $\mathfrak p_0(F_1')$,$\ldots$, $\mathfrak p_0(F_3'')$ as in \eqref{eq:DED'}.
Let us write the type $IV^*$ fiber as
\begin{equation} \label{eq-myfiber}
\Theta_1 + 2 \Theta_2 + \Theta_3 + 2 \Theta_4 + \Theta_5 + 2 \Theta_6 + 3 \Theta_0
\end{equation}
where $\Theta_{2i}$ meets exactly $\Theta_{2i-1}$ and $\Theta_0\; (i=1,2,3)$.
Recall that an additive fiber on an Enriques surface cannot be multiple. 
%
%
%
%
%
%

Start with the root $D_1=\mathfrak p_0(F_1')$.
By Tool \ref{tool:roots},
there is a composition of reflections $\mf_1'$ such that $\mf_1'(D_1) =\Theta_1$ as claimed.

%
%

Let $\mf_1=\mf_1'\circ\mf_0$ and  $\tilde{D}_{2}$  denote the effective divisor given by the decomposition of   $\mathfrak p_1(F_1'')$ 
defined in \eqref{eq:DED'}.
Since $\tilde D_2.\Theta_1=1$ and $0\leq \tilde D_2<2H$, we infer that $\Theta_1\not\subseteq \mbox{supp}(D_2)$
while $\Theta_2$ appears with multiplicity $1$ in $\tilde D_2$.
We claim that there are reflections in $\Theta_3,\hdots,\Theta_6, \Theta_0$ exclusively,
taking $\tilde D_2$ to $\Theta_2$.
To see this, we refer to the following more general property
which will be useful in the sequel, too.

\begin{Tool}
\label{tool:reflect}
Let $v$ be a root in a root lattice $R$
which contains the vertex $e$ with multiplicity one.
Then there is a composition $\mathfrak p$ of reflections in the other vertices of $R$
such that $\mathfrak p(v)=e$.
\end{Tool}

\begin{proof}
Denote the vertices of the Dynkin diagram of $R$ by $e_1,\hdots,e_n$
and write
\[
v=\sum_{j=1}^n a_j e_j \;\;\; (a_j\in\ZZ).
\]
Here $e=e_i$, say, and $a_i=1$ by assumption.
Since the roots in $R$ are always effective or anti-effective,
we infer $v\geq 0$ from $a_i$.
Since $v^2=-2$, there is some $j$ such that $v.e_j<0$.
If $j\neq i$, then the reflection $s_{e_j}$ reduces the complexity
of the root (measured in terms of $\sum_j a_j\geq 0$),
so we may continue with the root $s_{e_j}(v)\geq 0$ instead of $v$.

Assume that at some point during this process,
we have
\[
v.e_j\geq 0 \;\;\; \forall \, j\neq i.
\]
We claim that this implies $v=e_i$.
To see this, we compute
\[
-2 = v^2 = v.\sum_{j=1}^n a_je_j = -2 + \sum_{j\neq i} \underbrace{a_jv.e_j}_{\geq 0}.
\]
More precisely, the summands on the right-hand side are all zero if and only if $a_j=0$ for all $j$
with $e_j$ adjacent to $e_i$.
But then, since roots always have connected support, we infer that all $a_j$ for $j\neq i$ are zero as claimed.

In summary, we can apply reflections away from $e=e_i$ (reducing the complexity and preserving
effectivity) until $v$ is mapped to $e$ as stated.
\end{proof}

\begin{rem}
If  $e$ has coefficient $-1$ in the root $v\in R$,
then one can show analogously that reflections away from $e$ map $v$ to $-e$.
\end{rem}

Applied to $\tilde D_2$, we deduce that there is a reflection $\mathfrak p_2'$
such that
\[
\mathfrak p_2'(\tilde D_2) = \Theta_2, \;\;\; \mathfrak p_2'(\Theta_1) =\Theta_1.
\] 
That is, $\mf_2=\mf_2'\circ\mf_1$ maps $F_1', F_1''$ to $\Theta_1, \Theta_2$ up to multiples of $H$.

%
%

The same kind of reasoning applies to $F_2',\hdots,F_3''$ to show that
a composition of reflections
(in $\Theta_3,\hdots, \Theta_6$ only!) maps their image under $\mf_2$,
up to multiples of $H$,
to the fiber components $\Theta_3,\hdots, \Theta_6$.
The details are omitted for shortness.
This proves Lemma \ref{lem-1-2} for the first three fiber configurations from \eqref{eq:configs}.

\medskip

We now turn to the last two fiber configurations from \eqref{eq:configs}.
Here the configuration of 4 A$_2$'s is supported on a single fiber of Kodaira type $II^*$.
We shall seek to establish a contradiction to Lemma \ref{lem:l=3}, 
using Tools \ref{tool:roots}, \ref{tool:reflect}.

%

By Lemma \ref{lem:M}, there is a configuration of 3 A$_2$'s involving a 3-divisible class,
say
\[
\sum_{j=1}^{3} (\FFb'_{j} - \FFb''_{j})=3D.
\]
We start by embedding the remaining A$_2$-summand into the $II^*$ fiber
whose components we label as follows:

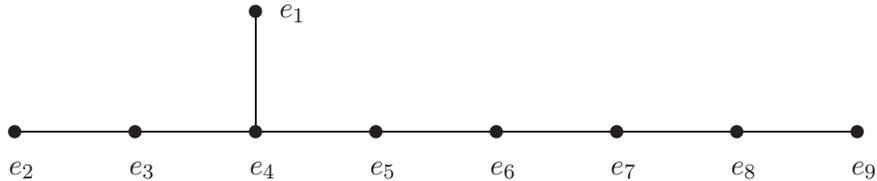
\begin{figure}[ht!]
\setlength{\unitlength}{.8mm}
\begin{picture}(100,30)(20,0)
\multiput(3,8)(20,0){8}{\circle*{2}}
\put(3,8){\line(1,0){140}}
\put(2,1){$e_2$}
\put(22,1){$e_3$}
\put(42,1){$e_4$}
\put(62,1){$e_5$}
\put(82,1){$e_6$}
\put(102,1){$e_7$}
\put(122,1){$e_8$}
\put(43,8){\line(0,1){20}}
\put(43,28){\circle*{2}}
\put(47,27){$e_1$}
\put(142,1){$e_9$}
\end{picture}
\end{figure}

We put  $D':=\mathfrak p_0(F_4')$ (resp. $D'' :=\mathfrak p_0(F_4''))$  
to denote the two  $(-2)$-divisors of the remaining A$_2$, both supported on the singular fiber
(effective or anti-effective by Claim \ref{claim}). 
By Tool \ref{tool:roots}, there is a composition $\mathfrak p_1'$ of reflections
mapping $D'$ to $e_9$.
Decomposing 
\[
\mathfrak p_1'(D'') = \tilde D + mH \;\;\; (0\leq \tilde D<2H)
\]
as before, we infer from the intersection number with $\mathfrak p_1'(D')=e_9 $ that $e_9\not\in\mbox{supp}(\tilde D)$
while $e_8$ has multiplicity one in $\tilde D$.
I.e.~$\tilde D\in E_8 = \tilde E_8\setminus\{e_9\}$,
and by Tool \ref{tool:reflect}, there is a composition $\mathfrak p_2'$ of reflections in $e_1,\hdots,e_7$
mapping $\tilde D$ to $e_8$.

It follows that  $\mathfrak p_2 = \mathfrak p_2'\circ\mathfrak p_1'\circ\mathfrak p_0$
maps
 $F_1'$, $\ldots$, $F_3''$  to effective or anti-effective $(-2)$-divisors in
the orthogonal complement of $\langle e_8, e_9\rangle$ inside the extended Dynkin diagram $\tilde E_8$
which is given by
\[
\langle e_8, e_9\rangle^\perp = \langle e_1,\hdots, e_6, 2H\rangle \cong \tilde E_6.
\]
In consequence, the above analysis of the $IV^*$ case applies verbatim to show that, after a suitable composition $\mathfrak p$
of reflections (in $e_1,\hdots,e_6$) and up to fiber multiples, 
the three A$_2$'s in question can be realized as
\begin{equation} \label{eq-fourforII}
\langle e_2, e_3\rangle \oplus \langle e_5,e_6\rangle \oplus \langle e_1, -(2e_1+e_2+2e_3+3e_4+2e_5+e_6)\rangle.
\end{equation}
Indeed, after mapping $\langle F_2',F_2'' \rangle$, $\langle F_3',F_3'' \rangle$ to 
$\langle e_2, e_3\rangle$, $\langle e_5,e_6\rangle$
and $F_1'$ to $e_1$, 
we solve the system of equations given 
by the fact that 
$\mathfrak p(F_1'')$ is a $(-2)$-divisor with prescribed  intersection pattern. 
This gives exactly the above solution modulo $2H$.
%
%
%

Apparently \eqref{eq-fourforII} does not give the signs of the six classes  $D_1'$, $\ldots$, $D_3''$ in the $3$-divisible divisor $\mathfrak p(3D)$.
Yet, since the  intersection numbers   $\mathfrak p(3D).e_j$ are multiples of $3$ for $j=1, \ldots,9$, 
one obtains only one possibility (up to sign and  a multiple of $(2\HE)$):
$$
(e_3-e_2) + (e_5-e_6) + (e_1+(2e_1+e_2+2e_3+3e_4+2e_5+e_6) = \pm 3  {\mathfrak p}(D)
$$
which yields 
\[
\mathfrak p(D) = \pm (e_1 + e_3 + e_4 + e_5)  + 2m \HE
\] 
for some $m\in\ZZ$.
In particular, $\mathfrak p(D)$ is either effective or anti-effective,
and applying $\mathfrak p^{-1}$, we infer the same for $D$ from Observation~~\ref{obs-strange}
(to be derived in \ref{ss:PL-K3}).
This contradicts Lemma~\ref{lem:l=3} and thus concludes the proof of Lemma \ref{lem-1-2}.
\end{proof}

\begin{rem}
\label{rem:3-torsion}
On the extremal rational elliptic surfaces,
the orthogonal A$_2$-configurations gives rise to $3$-torsion sections
by way of $3$-divisibility.
Essentially, this holds because H$^2(\RES,\ZZ)$ is unimodular.
Since the same applies to $\Num(\WW)$,
we will be able to establish the same results on $\WW$,
even though there is no section, see Lemma \ref{lem-f3333bigcode}.
\end{rem}

\begin{rem}
\label{rem:KZ}
The proof of \cite[Claim~3.5.1]{kz}
states that one can let go the fiber multiple in Lemma \ref{lem-1-2},
i.e.~
there is a composition of reflections $\mathfrak p_S$ such that
the image of each smooth rational curve $F_j', F_j''$ under $\mathfrak p_S$ is again represented by a smooth rational curve
(without possibly adding a multiple of $\HE$). 
Since we were not able to find a reference for this statement,
%
we decided to follow the advice of a referee and give a detailed proof of the weaker statement recorded in Lemma \ref{lem-1-2}
(which fortunately will be sufficient for our purposes).
\end{rem}


\subsection{Picard-Lefschetz reflections on the  K3-cover}
\label{ss:PL-K3}

In the last part of this section 
we study Picard-Lefschetz reflections on the  K3-cover $\XXX$ of $\WW$

Let $\proj : \XXX \rightarrow \WW$ be the K3-cover 
with induced elliptic fibration \eqref{eq:K3-ell} 
and let  $\deck \curvearrowright \XXX$ 
be the Enriques involution.
Given a smooth rational curve $\FF$ in $\WW$, 
  the preimage $\proj^{-1}(\FF)$ consists of two disjoint smooth rational curves 
$\FF^{+}, \FF^{-}$. 

With these preparations we proceed to the proof of Claim~\ref{claim}. 

\begin{proof}[Proof of Claim~\ref{claim}]
We maintain the notation of the proof of Lemma~\ref{lem-1-2}  (see \eqref{eq-defppzero}), and put 
$$
\mathfrak p_{0,Y} := (s_{E_l^{+}} \circ s_{E_l^{-}} \circ   
\ldots \circ s_{E_1^{+}}  \circ s_{E_1^{-}})  \, .
$$
This map is independent of the order of the elements of the pairs $E_i^+, E_i^-$
as we shall exploit below.
Let $D \in \mbox{Pic}(\WW)$. Observe that 
$(D.E_1) =  (\proj^{*}D.E_1^{+}) =  (\proj^{*}D.E_1^{-})$.
In particular, we have 
\begin{eqnarray}
(s_{E_1^{+}} \circ s_{E_1^{-}})(\proj^{*}D) & = &  \proj^{*}D
+ (\proj^{*}D.E_1^{+}) E_1^{+} +  (\proj^{*}D.E_1^{-}) E_1^{-} \nonumber \\
& =  &   \proj^{*}(D + (D.E_1)E_1) =  \proj^{*}(s_{E_1}(D)) \nonumber \, .
\end{eqnarray}
This yields  the equality 
\begin{eqnarray*}
\mathfrak p_{0,Y} \circ  \proj^{*} = \proj^{*} \circ  \, \mathfrak p_0. \label{eq-1-3} 
\end{eqnarray*} 
Similarly, one can show that
\begin{eqnarray*}
\mathfrak p_{0,Y} \circ  \deck^{*} = \deck^{*} \circ \mathfrak p_{0,Y} \, .  \label{eq-1-4}
\end{eqnarray*}  

Moreover, one has the equality
\begin{eqnarray}
\proj_{*}(\mathfrak p_{0,Y})(\FF^{+}) = \proj_{*}(\mathfrak p_{0,Y}(\FF^{-})) = {\mathfrak p_0}(\FF) \, .  \label{eq-1-5}
\end{eqnarray} 
Since $\mathfrak p_{0,Y}$ is an isometry, we have $\mathfrak p_{0,Y}(F_j'^{\pm})^2 = -2$. Therefore Riemann-Roch implies that either 
$|\mathfrak p_{0,Y}(F_j'^{\pm})| \neq \emptyset$ or $|-\mathfrak p_{0,Y}(F_j'^{\pm})| \neq \emptyset$. 
Suppose that $D' \in |\pm \mathfrak p_{0,Y}(F_j'^{\pm})|$. From
$(\mathfrak p_{0,Y}(F_j'^{\pm}).\proj^{*}\HE) = 0$, we infer that all components of $\mbox{supp}(D')$ are components of fibers of the fibration
$|\proj^{*}\HE|$. 
Claim~\ref{claim} follows now directly from \eqref{eq-1-5}. 
\end{proof}

Essentially, the above proof shows the following observation:

\begin{obs} \label{obs-strange}
If $D$ is an effective  $(-2)$-divisor on an Enriques surface, $\supp(D)$ consists of $(-2)$-curves, and $\mathfrak p$ is a composition of reflections in some $(-2)$-curves,
  then   $\mathfrak p(D)$ is either effective or anti-effective.
\end{obs}

Since the proof of  Lemma~\ref{lem-1-2} is now complete, we know that the map $\PLpsi_{\WW}$ (see \eqref{eq-plenr}) exists, and we  define 
\begin{equation} \label{eq-ppxxx}
\PLpsi_{\XXX} :=  (s_{E_k^{+}} \circ s_{E_k^{-}} \circ   
\ldots \circ s_{E_1^{+}}  \circ s_{E_1^{-}})  \, . 
\end{equation}
Obviously we have the equalities
$\PLpsi_{\XXX} \circ  \proj^{*} = \proj^{*} \circ \PLpsi_{\WW}$,  and   $\PLpsi_{\XXX} \circ  \deck^{*} = \deck^{*} \circ \PLpsi_{\XXX}$. \\
Moreover, it is immediate that
\begin{eqnarray*}
\proj_{*}(\PLpsi_{\XXX}(\FF^{+})) = \proj_{*}(\PLpsi_{\XXX}(\FF^{-})) = \PLpsi_{\WW}(\FF) \, .  
\end{eqnarray*} 
The latter implies using Zariski's Lemma 
that for $j=1, \ldots,4$, the divisor $\PLpsi_{\XXX}(\FFb_j'^{\pm})$ (resp. $\PLpsi_{\XXX}(\FFb_j''^{\pm})$)
is represented, up to sign, 
by a sum of 
smooth rational curves contained in a singular fiber of the elliptic fibration \eqref{eq:K3-ell} 
induced by $|\proj^{*}\HE|$ plus possibly a multiple of the general fiber of  $|\proj^{*}\HE|$.

In particular, $\XXX$ inherits 8 orthogonal A$_2$-configurations from $\WW$.
For later use, we label the curves 
$\FFb_j'^{\pm}$, $\FFb_j''^{\pm}$, 
in such a way that   
\begin{equation}  \label{eq-1-6}
\pi_*\PLpsi_{\XXX}(\FFb_j'^{\pm}) = \PLpsi_{\WW}\FFb_j'
 \;\,\mbox{ and } \;\, \pi_* \PLpsi_{\XXX}(\FFb_j''^{\pm}) = \PLpsi_{\WW}\FFb_j''
\;\,\mbox{ for } j = 1, \ldots, 4.
\end{equation}

In the sequel we will need the following simple observation. 
Suppose that for $j=1,\ldots,4$ the equalities
\begin{equation} \label{eq-useful-en}
\PLpsi_{\WW}\FFb_j' = \Theta_{2j-1} + n_{2j-1} \HE \quad \mbox{ and } \quad \PLpsi_{\WW}\FFb_j'' = \Theta_{2j} + n_{2j} \HE \, , 
\end{equation}
hold, where $\Theta_{2j-1}$, $\Theta_{2j}$ are components of singular fibers of $|2 \HE|$ and $n_{2j-1}, n_{2j} \in \ZZ$. Then, up to a relabelling of the 
rational curves  $\Theta_{2j-1}^{\pm}$, $\Theta_{2j}^{\pm}$, we have
\begin{equation} \label{eq-useful-K3}
\PLpsi_{\XXX}\FFb_j'^{\pm} = \Theta_{2j-1}^{\pm} + n_{2j-1}  \, \pi^{*}\HE \quad \mbox{ and } \quad \PLpsi_{\XXX}\FFb_j''^{\pm} =   \Theta_{2j}^{\pm} + n_{2j}  \, \pi^{*}\HE \, .
\end{equation}


\section{Two  families of Enriques surfaces with four cusps} \label{s-2fam}

In this section we construct families of Enriques surfaces with four disjoint A$_2$-configurations 
supported on the fibers of an elliptic fibration (following Lemma \ref{lem-1-2})
and study $3$-divisible sets on them.

\subsection{First family of Enriques surfaces}
\label{ss:1st}

Let $X_{3,3,3,3}$ be the extremal rational elliptic surface 
with four singular fibers  of the type $I_3$.
Locating them at
the third roots $\mu_3$ of $(-1)$ and at $\infty$,
the surface is given by the Hesse pencil
\[
X_{3,3,3,3}: \;\;\; x^3+y^3+z^3 + 3\lambda xyz = 0.
\]
Here the 3-torsion sections alluded to in Remark \ref{rem:3-torsion}
enter as the base points of the cubic pencil.
An Enriques surface is obtained from $X_{3,3,3,3}$ 
by applying logarithmic transformations of order $2$
 to the elliptic fibers over two distinct points $P_1, P_2  \in \PP^1$.
As explained in \ref{ss:ell},
this depends on the choice of
$2$-torsion points in the fibers of $\Xt$ over $P_1$, $P_2$.
However,  this subtlety will not cause us any trouble:

\begin{lemm}
\label{lem:Ft}
The Enriques surfaces obtained by a logarithmic transformation of order $2$ from $X_{3,3,3,3}$ as above
form an irreducible two-dimensional family $\Ft$.
\end{lemm} 

\begin{proof}
The moduli space $\Ft$ is a degree $9$ ramified covering of the configuration space
\begin{eqnarray}
\label{eq:covering}
\mbox{Sym}^2 (\PP^1) \setminus \text{diagonal}.
\end{eqnarray}
Thus $\Ft$ has dimension two.
Here the fiber above a non-ordered pair $\{P_1, P_2\}$ consists of all non-ordered pairs 
in
\[
(E_{P_1}[2]\setminus \{O\}) \times (E_{P_2}[2]\setminus \{O\})
\]
where $E_P$ denotes the fiber of $X_{3,3,3,3}$ at $P\in\PP^1$ with zero element $O$ and two-torsion subgroup $E_P[2]$.
It follows that the covering \eqref{eq:covering} ramifies exactly at the singular fibers.
Since the monodromy action of 
\[
\pi_1(\mbox{Sym}^2 (\PP^1\setminus (\mu_3\cup\{\infty\})) \setminus \text{diagonal})
\] 
on a general fiber is transitive,
the moduli space $\Ft$ is irreducible.
\end{proof}

This proves the first part of Theorem \ref{thm0}.
%
In view of Lemma \ref{lem:Ft}, 
we will allow ourselves to abuse notation and denote the resulting Enriques surface(s) simply by 
$\WW_{P_1,P_2}$. 
%
%

Now, let $\WW = \WW_{P_1,P_2} \in {\mathcal F}_{3,3,3,3}$ be an Enriques surface with  K3-cover $\XXX$
and elliptic fibrations $\varphi, \tilde\varphi$ in the notation of \ref{ss:ell}.
%
%
%
We continue by establishing some information about $\XXX$ with the help of $\Jac(\XXX)$.
As far as 
$P_1, P_2$ do not hit $\infty$ and third roots of $(-1)$
(i.e. outside the branch locus of \eqref{eq:covering}), 
we obtain
 eight fibers of the type  $I_3$ on $\XXX$ and $\Jac(\XXX)$, 
 so  the Picard number $\rho(\XXX)=\rho(\Jac(\XXX))$ is at least $18$ by the Shioda-Tate formula. 
 
 \begin{lemm}
 \label{lem:324}
 If $\rho(\XXX)=18$, then $\NS(Y)$ has discriminant 
 $\operatorname{d}(\mbox{NS}(\XXX)) = -324$.
 \end{lemm}
 
 \begin{proof}
 By assumption, $\Jac(\XXX)$ has finite Mordell-Weil group.
 The configuration of singular fibers only accommodates 3-torsion,
 so we infer 
 \[
 \MW(\Jac(\XXX)) \cong (\ZZ/3\ZZ)^2
 \]
 by pull-back from $\Xt$.
 Hence 
 $\operatorname{d}(\mbox{NS}(\Jac(\XXX))) = -81$. 
By the existence of a bisection on $\XXX$
(induced from $\WW$, see \ref{ss:ell}), we infer from \cite[Lemma~2.1]{keum}  that  
\begin{equation} \label{eq-cover-3333}
\mbox{ either } \operatorname{d}(\mbox{NS}(\XXX)) =  - 81 \mbox{ or } \operatorname{d}(\mbox{NS}(\XXX)) = -324
\end{equation}
as soon as $\rho(\XXX) = 18$.
(Here the former equality holds iff  $\tilde{\elpi}$ admits a section
i.e. iff $\XXX=\Jac(\XXX)$.) 
Lemma \ref{lem:324} now
results immediately from the following proposition.
 \end{proof}
%

\begin{prop} \label{obs-PicK3cover}
Let $Y$ be the K3-cover of an Enriques surface. 
Then 
\[
2^{20-\rho(Y)}\mid \operatorname{d}(\mbox{NS}(Y)).
\]
In particular, if $\operatorname{d}(\mbox{NS}(Y))$ is odd, then
$\rho(Y) = 20$.
\end{prop}

\begin{proof} 
We shall use the  primitive embedding
$$
L:= U(2) + E_8(2) \cong \pi^*\Num(\WW) \hookrightarrow \mbox{NS}(Y) \, .
$$
We follow the notation of \cite[$\S$~5$^{\circ}$]{nikulin} 
and denote the discriminant group of $L$ by 
\[
\mbox{A}_{L}:=L^\vee/L;
\]
likewise for other primitive sublattices of $\NS(Y)$ such as $L^\perp$.
Define the finite abelian group
$$H :=  
\mbox{NS}(Y) / {(L \oplus L^{\perp})}.$$
Obviously we have the inclusion
$$
H \subset \mbox{A}_{L} \oplus  \mbox{A}_{L^{\perp}}.
$$
Let $p_L$ (resp. $p_{L^{\perp}}$) be the projection 
from $\mbox{A}_{L} \oplus  \mbox{A}_{L^{\perp}}$ onto the first (resp. the second) summand.
By  \cite[p.~111]{nikulin} either projection is an embedding.
The first embedding implies 
\[
H \cong (\ZZ/2\ZZ)^{l},
\] 
while the second shows $l \leq \rho - 10$ since the length of $A_{L^\perp}$
is bounded by the rank of $L^\perp$. 
We obtain
$$ 
\operatorname{d}(\mbox{NS}(Y)) = \operatorname{d}(L \oplus L^{\perp})/|H|^2 = 2^{10-l} \cdot (\operatorname{d}(L^{\perp})/2^{l}).
$$
Note that the right-most term in brackets is an integer 
since $|H|=|p_{L^{\perp}}(H)|$ divides $|\mbox{A}_{L^{\perp}}|=\operatorname{d}(L^{\perp})$.
Hence we infer that  $2^{20 - \rho} | \operatorname{d}(\mbox{NS}(Y))$ as claimed.
\end{proof}

\begin{rem}
A detailed analysis using the $2$-length of the groups involved
allows one to strengthen the above line of arguments to prove that
the K3 cover $\XXX$ of an Enriques surface has $A_{\NS(\XXX)}$
of $2$-length at least $20-\rho(\XXX)$.
\end{rem}

\subsection{3-divisible sets}

We shall now investigate the 3-divisible sets among the 4 A$_2$-configurations supported on fibers
of an Enriques surface $S\in\Ft$.
Our main results will be formulated in Lemma \ref{lem-f3333bigcode} and Lemma \ref{lem-f3333smallcode}.

Let $G$ be  a $2$-section of the elliptic fibration $\varphi$ and let $F_j, F_j', F''_j$, where $j= 1, \ldots 4$,
be  the components of the $I_3$-fibers of $\varphi$. 
In order to streamline our notation we label the components of the singular fibers in the following way
relative to $G$:

\begin{notation}
If $G$ meets only one
component of an $I_3$-fiber we denote  this component by $F_j$. Otherwise, $F_j', F_j''$ stand for the components
of the  $I_3$-fiber that meet the $2$-section $G$ (i.e. we have $G.F_j = 0$ then). 
\end{notation}

In particular, if 
$(F_j + F_j' + F''_j)$ happens to be a half-pencil of the fibration in question, we assume that $G.F_j =1$.
After those preparations we can study $3$-divisible sets in the fibers of the elliptic fibration $\varphi$ on $\WW$
and $\tilde{\varphi}$ on $\XXX$.

 \begin{lemm} \label{lem-f3333bigcode}
Let $\WW \in {\mathcal F}_{3,3,3,3}$. 
The A$_2$-configurations 
\begin{equation} \label{eq-eqscbigen}
\FFb'_1, \FFb''_1, \ldots \FFb'_4, \FFb''_4
\end{equation}
 contain four $3$-divisible sets.
\end{lemm}

\begin{proof}
By \eqref{eq-unimod} 
it suffices to prove that $\overline M$,
the primitive closure of the lattice $M$ 
spanned in $\mbox{Num}(\WW)$ by the curves \eqref{eq-eqscbigen},
 is unimodular.
Equivalently, the lattice
$M^\perp=\overline M^\perp$ is unimodular.
To see this,
define an auxiliary divisor class
$$
D  := G + \sum_{\{j: G.F_j=0\}}  (F'_j + F''_j) \in M^\perp.
$$
Let $B$
denote  a half-pencil of the fibration $\varphi$.
By construction, $B\in M^\perp$,
and $B,D$ span the hyperbolic plane $U$
since 
$D.B = G.B =1$ and $B^2 = 0$. Thus $M^\perp$ and  $\overline M$ are unimodular, and
 the proof of Lemma \ref{lem-f3333bigcode} is completed by \eqref{eq-unimod}.
\end{proof}

We shall now eliminate all but one 3-divisible classes by considering a different configuration
of 4 A$_2$'s on $S\in\Ft$.
Recall that  $F_j^{+}$, $F_{j}^{-}$ stand for  the $(-2)$-curves on the   K3-cover 
$\pi: \XXX \rightarrow \WW$ that lie over the smooth rational curve $F_j$, and likewise for $F'_j, F''_j$.
A discussion of properties of $3$-divisible sets of A$_2$-configurations on K3 surfaces can be found in  \cite{b}.
In particular, by \cite[Lemma~1]{b},  a 
$3$-divisible set of A$_2$-configurations on a~K3~surface consists always of six or nine  such configurations.

 \begin{lemm} \label{lem-f3333smallcode}
Let $\WW \in {\mathcal F}_{3,3,3,3}$. Then  

\begin{enumerate}
\item[(a)] The four  A$_2$ configurations 
\begin{equation} \label{eq-eqsc3sm}
F'_1, F_1'', \ldots, F'_3, F_3'',F_4' ,F_4
\end{equation}
on the Enriques surface $\WW$ contain exactly one $3$-divisible set. 

\item[(b)] If the I$_3$ configuration $(F_4 + F_4' + F_4'')$ is not a half-pencil, then the eight A$_2$ configurations 
\begin{equation} \label{eq-eqsc3smk3}
F_1'^{+}, F_1''^{+}, F_1'^{-}, F_1''^{-}, \ldots, F_3'^{-}, F_3''^{-}, F_4'^{+},  F_4^{+}, F_4'^{-},  F_4^{-}
\end{equation}
on the K3-cover $\XXX$ contain exactly one  $3$-divisible set.
\end{enumerate}
\end{lemm}

\begin{proof}  
{\sl (a):} By \eqref{eq-2poss} we are to show  that \eqref{eq-eqsc3sm} does not contain four  $3$-divisible sets.
Suppose to the contrary. Then each triplet of A$_2$-configurations in \eqref{eq-eqsc3sm} is $3$-divisible. In particular, we have 
$$
\sum_{j=2}^{3}(\lambda'_j \FFb'_j + \lambda_j'' \FFb_j'') + \lambda_4 \FFb_4 + \lambda'_4 \FFb'_4  = 3{\mathcal L} \, , \quad \mbox{ where } \{\lambda'_j, \lambda''_j\} = \{\lambda_4, \lambda'_4\} = \{1, -1 \} \, .
$$
Since  $G.(\lambda'_j \FFb'_j + \lambda_j'' \FFb_j'')=0$ for $j=2,3$, we obtain $G.(\lambda_4 \FFb_4 + \lambda'_4 \FFb'_4) \in 3\ZZ$.

If $G$ meets only the curve $\FFb_4$ in the fiber $(F_4 + F_4' + F''_4)$ (resp.   $2(F_4 + F_4' + F''_4)$ iff we deal with a half-pencil)
we have $G.\FFb_4 \in \{2,1\}$ and  $G.\FFb'_4 = 0$, so $\lambda_4 \in 3\ZZ$. Contradiction. 

Otherwise,  $G$ meets the fiber $(F_4 + F_4' + F''_4)$ in two different components, i.e. 
$G.\FFb'_4 = G.\FFb''_4 = 1$  and  $G.\FFb_4 = 0$, which yields 
$\lambda_4' \in 3\ZZ$. Again we arrive at a contradiction, which  implies by symmetry and Lemma \ref{lem:M}
that 
\begin{equation} \label{eq-strrr}
\FFb'_1, \FFb_1'', \FFb'_2, \FFb_2'',  \FFb'_3, \FFb_3'' \mbox{ form the unique } 3 \mbox{-divisible set in \eqref{eq-eqsc3sm}. } 
\end{equation}

{\sl (b):} Since the pull-back of a (non-trivial) $3$-divisible divisor under $\pi$ is (non-trivially) $3$-divisible,   \eqref{eq-strrr} implies that  the six   A$_2$-configurations 
\begin{equation}
F_1'^{+}, F_1''^{+}, F_1'^{-}, F_1''^{-}, \ldots ,F_3'^{+}, F_3''^{+}, F_3'^{-}, F_3''^{-} 
\end{equation}
are 
$3$-divisible on the K3-cover $Y$. To show that they form the unique $3$-divisible configuration in \eqref{eq-eqsc3smk3},
assume that the A$_2$-configuration $F_4^{+}$, $F_4'^{+}$ is contained in another non-trivial $3$-divisible set on $Y$. 

Suppose that neither the curve  $F_4^{-}$ nor $F_4'^{-}$ is  contained in the $3$-divisible divisor in question. 
Since  
 $\pi$ is unramified,
 push-forward yields a non-trivial $3$-divisible set of three A$_2$-configurations 
in \eqref{eq-eqsc3sm} that contains  $F_4$, $F_4'$. The latter is impossible by \eqref{eq-strrr}. 

Thus we can assume that the curves 
$F_4^{-}$, $F_4'^{-}$, $F_4^{+}$, $F_4'^{+}$ are contained in the support of the $3$-divisible divisor in question. 
From the properties of the push-forward $\pi_{*}$ and \eqref{eq-strrr},
we infer the existence of  $\lambda_j'^{\pm}$,  $\lambda_j''^{\pm}$  $\in \{0, 1, -1 \}$, 
such that one has 
\begin{equation*}
\sum_{j=1}^{3}(\lambda_j'^{+} \FFb_j'^{+} + \lambda_j'^{-} \FFb_j'^{-} 
+ \lambda_j''^{+} \FFb_j''^{+} + \lambda_j''^{-} \FFb_j''^{-}) +
(\FFb_4'^{-} - \FFb_4^{-}) -  (\FFb_4'^{+} - \FFb_4^{+}) = 3\tilde{\mathcal L}
\end{equation*}
for a divisor $\tilde{\mathcal L}$ on $\XXX$.
By Lemma~\ref{lem-f3333bigcode}, each triplet of A$_2$-configurations in \eqref{eq-eqscbigen}  is $3$-divisible, so we can assume that for $j=2,3,4$
there exist $\mu'_j,   \mu_j''$,  such that  
\begin{equation*}  \label{eq-sr}
\sum_{j=2}^{4}(\mu'_j (\FFb_j'^{+} + \FFb_j'^{-}) + \mu_j'' (\FFb_j''^{+} + \FFb_j''^{-})) = 3\hat{\mathcal L}  \mbox{ and } \{\mu'_j,   \mu_j''\} = \{1,-1\}
\end{equation*}
for some $\hat{\mathcal L}\in\Pic(\XXX)$. 
After exchanging components, if necessary, we can assume that $(\mu'_4,   \mu_4'') = (1,-1)$. 
By adding the previous two equalities we arrive at a 3-divisible divisor  
$$
D + (2 \FFb_4'^{+} + \FFb_4''^{+}) + 3( \FFb_4'^{-} + \FFb_4^{+}) - (\FFb_4^{-} + \FFb_4'^{-} + \FFb_4''^{-}) - 2 (\FFb_4^{+} + \FFb_4'^{+} + \FFb_4''^{+})
$$
with
$\mbox{supp}(D)$ contained in the union of the curves  $\FFb_j'^{\pm}$, $\FFb_j''^{\pm}$ for $j = 1, 2, 3$.
Since both  triangles $(\FFb_4^{-} + \FFb_4'^{-} + \FFb_4''^{-})$, $(\FFb_4^{+} + \FFb_4'^{+} + \FFb_4''^{+})$
are fibers of the elliptic fibration $\tilde{\varphi}$,
 we derive a 3-divisible divisor
\begin{equation}  \label{eq-ffr}
D - ( \FFb_4'^{+} - \FFb_4''^{+})
\end{equation}
with
$\mbox{supp}(D)$ satisfying the condition given above.
We continue to establish a contradiction.


Recall (see e.g. \cite[$\S$~5]{r}) that each  non-trivial $3$-divisible set  on $\XXX$  corresponds to a line $\FFF_3 v$, where $v$ is a non-zero vector in the kernel of the 
$\FFF_3$-linear map
\begin{equation*} 
 \FFF_3^8 \ni (\lambda_1^{+},\ldots, \lambda_4^{-}) \mapsto  \sum_{1}^{4} \lambda_j^{+} (\FFb_j'^{+} - \FFb_j''^{+}) + \sum_{1}^{4} \lambda_j^{-} (\FFb_j'^{-} - \FFb_j''^{-}) \in 
\mbox{Pic}(\XXX) \otimes \FFF_3.
\end{equation*}
Thus the kernel in question is a ternary $[8,d,\{6\}]$-code (i.e. a $d$-dimensional subspace of $\FFF_3^8$, such that all its  non-zero vectors have exactly $6$ non-zero coordinates). 
By the Griesmer bound
(see e.g. \cite[Thm (5.2.6)]{vlt}), we have $d \leq 2$ and $\FFb_j'^{\pm}$, $\FFb_j''^{\pm}$,  where $j = 1, \ldots, 4$, contain at most four sets of $3$-divisible  A$_2$-configurations.
On the other hand we obtain four non-trivial $\psi^{*}$-invariant $3$-divisible sets  by pulling-back the  $3$-divisible sets from $\WW$ (see Lemma~\ref{lem-f3333bigcode}). Observe that the $3$-divisible set
given by  \eqref{eq-ffr} is not $\deck^{*}$-invariant. Contradiction.   
\end{proof}

%

\begin{rem}
\label{rem:I_3}
Since any elliptic fibration on an Enriques surface has exactly 2 multiple fibers,
we can always ensure by exchanging fibers that the assumption in Lemma \ref{lem-f3333smallcode} (b) holds.
\end{rem}

\subsection{Second family of Enriques surfaces}
\label{ss:2nd}

In the following paragraphs,
we work out  Enriques surfaces with elliptic fibrations of the types \eqref{eq-nppss}.
To this end, we consider another two extremal rational elliptic surfaces, given in Weierstrass form
\begin{eqnarray*}
\RES_{4,3,1}: & y^2 + xy + ty = x^3;\\
\RES_{4,4}: & y^2 + ty = x^3.
\end{eqnarray*}
Each has a fibre of Kodaira type $IV^*$ at $\infty$ and a 3-torsion section at $(0,0)$.
The fibre type at $t=0$ is $IV$ for $X_{4,4}$ and $I_3$ for $X_{4,3,1}$ (which thus has one further singular fiber, of type $I_1$).
The surface $X_{4,4}$ gives the special case in \eqref{eq-nppss} omitted in \cite{kz}.
We point out that both elliptic surfaces live inside an isotrivial family
\begin{eqnarray}
\label{eq:isotrivial}
\mathcal X:  y^2 + cxy + ty = x^3, \;\; c\in\CC
\end{eqnarray}
whose fibers $\mathcal X_c$ for $c\neq 0$ are all isomorphic to $\Xf$ after rescaling
while $\mathcal X_0=X_{4,4}$.
Implicitly, this connection will feature again shortly.

As in \ref{ss:1st},
we shall apply logarithmic transformations of order $2$ to $\Xf$ and $X_{4,4}$.
On $\Xf$, we can do so for any two distinct  points $P_1, P_2 \in \PP^1 \setminus \{ \infty\}$;
again, the logarithmic transformation depends on the choice of 2-torsion points in the fibers,
but as in the proof of Lemma \ref{lem:Ft}, we obtain an irreducible 2-dimensional family $\mathcal F$
of Enriques surfaces, parametrized by a 9-fold covering of the configuration space
$\mbox{Sym}^2 (\PP^1\setminus\{\infty\}) \setminus \text{diagonal}$.

The situation is quite different for the logarithmic transformations of $X_{4,4}$:
since $X_{4,4}$ has only two singular fibers (and the above Weierstrass form has only three terms), 
we can still rescale the base curve $\PP^1$
without changing the surface at all!
Hence the resulting Enriques surfaces do only come in a one-dimensional family.

\begin{lemm}
\label{lem:boundary}
The Enriques surfaces arising from $X_{4,4}$ via logarithmic transformation
lie in the boundary of $\mathcal F$.
\end{lemm}

We did not attempt to prove Lemma \ref{lem:boundary}
directly from the above data (although the isotrivial family \eqref{eq:isotrivial} certainly points in this direction).
Instead, we will pursue an alternative algebraic approach towards the Enriques surfaces with fibers of type $IV^*$ and 
$I_3$ or $IV$ in Section \ref{s:moduli}.
Incidentally, this will provide an easy proof of Lemma \ref{lem:boundary}, see Remark \ref{rem:j=0}.

\begin{defi}
We denote the resulting irreducible family of Enriques surfaces by $\Ff$
(comprising surfaces arising from $\Xf$ as well as from $X_{4,4}$).
\end{defi}

%

%
On an Enriques surface $\WW' 
\in {\mathcal F}_{4,3,1}$, we put
\begin{equation} \label{eq-4star} \nonumber 
3 \FFb_0 + \sum_{j=1}^{3} (\FFb'_j + 2 \FFb''_j)
\end{equation}
(resp.  $\FFb_4 + \FFb'_4 + \FFb''_4$)  
to denote the $IV^{*}$-fiber (resp. the $I_3$ or $IV$-fiber) of  the induced elliptic fibration $\elpi$.
It is immediate that, up to the choice of the curve $\FFb_4$, 
 the rational curves  
$$
\FFb'_1, \FFb_1'', \ldots \FFb'_4, \FFb_4''
$$
 form the only set of four disjoint
 A$_2$-configurations contained in the singular fibers of the fibration $\elpi$.

Let $\pi \, : \, \XXX' \rightarrow \WW'$ be the K3-cover and let $\tilde{\varphi}$ be the fibration induced by $\varphi$ on $ \XXX'$. 
The number of $3$-divisible sets on $\WW'$ (resp. $\XXX'$)  supported on the components of fibers of $\varphi$ (resp. $\tilde{\varphi}$)
 can be found using \cite[Lemma~3.5~(1)]{kz}. 
 
\begin{lemm} \label{lem-f431} 
Let $\WW' \in {\mathcal F}_{4,3,1}$.
Then the four A$_2$ configurations 
$$
\FFb'_1, \FFb_1'', \ldots \FFb_4', \FFb_4''
$$ 
contain exactly one $3$-divisible set, whereas 
the eight A$_2$-configurations 
$$
{\FFb'_1}^{+}, \FFb_1''^{+}, \ldots, \FFb_4'^{-} ,\FFb_4''^{-}
$$ 
on the K3-cover contain four $3$-divisible sets.
\end{lemm}

\begin{proof}
By \cite[Lemma~3.5~(1)]{kz} the set   $\FFb_1', \FFb_1'', \ldots, \FFb_3', \FFb_3''$ is not 
$3$-divisible, 
whereas the six A$_2$-configurations 
$\FFb_1'^{+}, \FFb_1''^{+}, \FFb_1'^{-}, \FFb_1''^{-}, \ldots, \FFb_3'^{-}, \FFb_3''^{-}$ 
form a $3$-divisible set on $\XXX'$
(pushing down to a trivially 3-divisible divisor on $\WW'$). 
The former assertion
rules out the second possibility of Corollary \ref{eq-2poss}. 

On the other hand,
the curves $\FFb'_1, \FFb_1'', \ldots \FFb'_4, \FFb_4''$ contain at least one $3$-divisible set by Corollary \ref{eq-2poss}. 
As the pullback under $\pi$  we obtain another $3$-divisible
set on $\XXX'$, so the K3-cover contains exactly four $3$-divisible sets (see the proof of Lemma~\ref{lem-f3333smallcode}).  
\end{proof}


\subsection{Proof of Theorem \ref{thm0}}

Let $S$ be an Enriques surface
admitting 4 disjoint A$_2$-configurations.
Then $S$ admits an elliptic fibration 
\[
\pi: S \to \PP^1
\]
whose singular fibers are classified in Lemma \ref{lem-1-2}.
The Jacobian of $\pi$ is either $\Xt, \Xf$ or $X_{4,4}$
(as exploited in the proof of Lemma \ref{lem-1-2}).
In particular, $S$ arises from $\Jac(\pi)$ by a logarithmic transformation of degree $2$.
Thus $S\in\Ft\cup\Ff$ as shown in \ref{ss:1st}, \ref{ss:2nd}.
\qed

\subsection{Explicit examples supporting each case}

\begin{examp}  \label{example-kondo} 
Let $\WW'$ be the Enriques surface with finite automorphism group
$S_4\times\ZZ/2\ZZ$ considered in \cite[Example V]{kondo}
arising from the Kummer surface of $E^2$ for the elliptic curve with zero j-invariant. 
By \cite[Table~2 on p.~132]{kondo}
we have 
$$
\WW' \in {\mathcal F}_{4,3,1} \setminus {\mathcal F}_{3,3,3,3}\, .
$$
\end{examp}

In fact, by \cite[Remark~(4.29)]{kondo} no surface in  $ {\mathcal F}_{3,3,3,3}$ has finite automorphism group, 
so we can find no example of a surface from  ${\mathcal F}_{3,3,3,3}$ in \cite{kondo}. Therefore, we will use 
the construction of Enriques involution of base change type
(as reviewed in \ref{ss:ell})            
 to obtain an explicit example of such an Enriques surface.
 
 \begin{examp}
 \label{ex-2}
 Let $\XXX$ denote the singular K3 surface with transcendental lattice
 \begin{eqnarray}
 \label{eq:TT}
 T(\XXX) =
 \begin{pmatrix}
 6 & 3\\
 3 & 12
 \end{pmatrix}.
 \end{eqnarray}
 In what follows, we will sketch in a rather conceptual way
 that $\XXX$ admits an Enriques involution of base change type
 whose quotient surface is in $\Ft$.
 
 We start from elliptic curves $E$ parametrised by the j-invariant.
 To $E^2$, we can associate the Kummer surface $\mbox{Kum}(E^2)$,
 but also by way of what is called a Shioda-Inose structure nowadays (cf.~\cite{SI}) a K3 surface 
 which recovers the transcendental lattice of $E^2$.
 We thus obtain a one-dimensional family of K3 surfaces $\mX$
 with generic transcendental lattice
 \[
 T(\mX) = U+\langle 2\rangle.
 \]
 By Nishiyama's method \cite{nishi} $\mX$ comes with an elliptic fibration with a fiber of Kodaira type $I_{18}$
 and generically $\MW(\mX)\cong \ZZ/3\ZZ$.
 Quotienting out by translation by the $3$-torsion sections,
 we obtain another family of K3 surfaces $\mY$,
 generically with one fiber of type $I_6$, 6 fibres of type $I_3$ and 
 \[
 \MW(\mY)\cong(\ZZ/3\ZZ)^2.
 \]
 It follows that $\mY$ arises from $\Xt$ by the one-dimensional family of base changes \eqref{eq:base-change}
 ramified at a given singular fiber, say $\lambda=\infty$.
 That is, there is an involution $\imath\curvearrowright\mY$
 such that $\mY/\imath = \Xt$.
 For discriminant reasons, the transcendental lattice is scaled by a factor of $3$
 \begin{eqnarray}
 \label{eq:T(3)}
 T(\mY) \cong T(\mX)(3)
 \end{eqnarray}
 where this equality not only holds generically, but also on the level of  single members of the families 
 (cf.~ \cite[Lemma~8]{schuett-izvestiya}).
 
 We continue by specialising to a member $\XXX\in \mY$
 in order to endow $\XXX$ with a section
 which combines with $\imath$ to an Enriques involution of base change type. 
 To this end, choose $E$ to be the elliptic curve with CM by $\ZZ[\omega]$ for $\omega=(1+\sqrt{-7})/2$
 (j-invariant $-15^3$).
By \cite{SM},
\[
T(E^2)
\cong \begin{pmatrix}
2 & 1\\1 & 4\end{pmatrix}
\] 
which exactly gives rise to \eqref{eq:TT} by \eqref{eq:T(3)}.
 Inside the family $\mY$, this can only be accounted for by a section $Q$ of height $7/6$.
 It is induced from a section $Q'$ of height $7/12$ on the quadratic twist $\Xt'$ of $\Xt$.
 Here $\Xt'$ has singular fibers of types $I_3^*$, 3 times $I_3$ and $I_0^*$;
 the given height can only be attained if $Q'$ is perpendicular to $O'$
 and intersects nontrivially $I_3^*$ (far simple component), one $I_3$ and $I_0^*$.
 In consequence, the pull-back $Q$ on $\XXX$ is disjoint from $O$ and anti-invariant for $\imath$.
 Hence 
 \[
\jmath:=  (\text{translation by } Q) \circ\imath
 \]
 defines an Enriques involution on $\XXX$ such that $\XXX/\jmath\in\Ft$ as claimed.
 All of this can be made explicit without much difficulty.
 For instance, spelling out the conditions for $Q'$ on $\Xt'$
 to take the above shape,
 one finds that the second ramification point of the quadratic base change \eqref{eq:base-change}
 is located at $\lambda=5/4$.
 We leave the details to the reader.
 \end{examp}

 \subsection{Gorenstein $\QQ$-homology projective planes}
 
 In \cite{HKO},
 a classification of Gorenstein $\QQ$-homology projective planes
 is completed in terms of their singularity types.
 A key case originates from Enriques surfaces after contracting 
a set of nine $(-2)$-curves.
Among the 31 singularity types which are a priori possible, only
two are so far not supported by an example.
Here we note that the Enriques surface $S=Y/\jmath$ from Example \ref{ex-2}
remedies this for one type.
(See also \cite{S-fake} for subsequent results in the same spirit.)

\begin{prop}
\label{prop}
$S$ contains an A$_3+3$A$_2$ configuration of smooth rational curves.
\end{prop}
 
\begin{proof}
By construction, $S$ is equipped with an elliptic fibration
with 4 singular fibers of Kodaira type $I_3$,
the one at $\infty$ actually with multiplicity two.
Consider the bisection $R$ whose pre-image decomposes into $O$ and $Q$ 
on the covering K3 surface.
Since $O$ and $Q$ are disjoint, $R$ is a smooth rational curve,
and $R^2=-2$.
By the set-up in Example \ref{ex-2},
$R$ meets all but one singular fiber in a single fiber  component.
Therefore there are three A$_2$ configurations supported on the fibers
which are perpendicular to $R$.
The remaining $I_3$ fiber connects with $R$ for a square with one diagonal added. 
Omitting one fiber component meeting $R$, we obtain an A$_3$ configuration.
\end{proof}

\section{The fundamental groups of open Enriques surfaces} \label{s-fundamental groups}

Let $\WW$ be an Enriques surface that   
contains four disjoint A$_2$-configurations $\FFb_1'$, $\FFb_1'', \ldots$, $\FFb_4'$, $\FFb_4''$ and let  
$\pi \, : \, \XXX \rightarrow \WW$ be the K3-cover. 
As in the preceeding sections, 
the $(-2)$-curves in $\pi^{-1}(\FFb'_j)$ are denoted by $\FFb_j'^{+}$, $\FFb_j'^{-}$.
Moreover,  we put
\begin{equation*}
{\mathcal A} := \{ \FFb'_1, \FFb_1'', \ldots \FFb_4', \FFb_4'' \} \mbox{ and }   {\mathcal A}^{\pm} := \{ {\FFb'_1}^{+}, \FFb_1''^{+},{\FFb'_1}^{-}, \FFb_1''^{-}, \ldots,\FFb_4'^{-}, \FFb_4''^{-} \}.
\end{equation*}
Given the pair $(\WW, {\mathcal A})$, we follow \cite{kz} and define
the {\sl fundamental group  of the open Enriques surface $\WW^\circ = \WW \setminus \mathcal A$}:
\begin{equation*} \label{eq-fgdef}
\pi_1(\WW, {\mathcal A}) :=  \pi_1(\WW^\circ). 
\end{equation*}

To deal with Enriques surfaces with four A$_2$-configurations in more generality we introduce the following notation: 

\begin{notation}
 We say that 
the pair $(\WW, {\mathcal A})$ belongs to ${\mathcal F}_{4,3,1}$ (resp. ${\mathcal F}_{3,3,3,3}$)  
iff there exists a composition of Picard-Lefschetz reflections 
\eqref{eq-plenr} and an elliptic pencil $|2 \HE|$ on $\WW$ such that  the elliptic fibration given by  $|2 \HE|$ has singular fibers of the types \eqref{eq-nppss} (resp. \eqref{eq-ppss})
and,  up to multiples  of the half-pencil $\HE$, each class $\PLpsi_{\WW}(\FFb'_j)$, $\PLpsi_{\WW}(\FFb''_j)$, where $j=1,\ldots,4$, 
is an irreducible component of a singular fiber  of the elliptic fibration   $|2 \HE|$.
To simplify our notation we write
$$
(\WW, {\mathcal A}) \in {\mathcal F}_{4,3,1} \;\; (\mbox{resp.} \;(\WW, {\mathcal A}) \in {\mathcal F}_{3,3,3,3})
$$
when the above condition is satisfied. Then $\PLpsi_{\WW}({\mathcal A})$ stands for the set of the four A$_2$-configurations
defined, up to multiples of $H$, 
by $\PLpsi_{\WW}(\FFb'_1)$, $\ldots$, $\PLpsi_{\WW}(\FFb''_4)$ for a fixed composition
of reflections $\PLpsi_{\WW}$. 
\end{notation}

Recall that $\PLpsi_{\WW}$ induces the map  $\PLpsi_{\XXX}$ (see \eqref{eq-ppxxx}). 
In the sequel we maintain the notation \eqref{eq-1-6} and 
use  $\PLpsi_{\XXX}({\mathcal A}^{\pm})$ to denote the set of the eight  A$_2$-configurations on the K3-cover
(again supported on the fibers of an elliptic fibration).

As we explained around Lemma \ref{lem-1-2}, the authors of \cite{kz} claim 
that after applying an appropriate composition of Picard-Lefschetz
reflections $\PLpsi_{\WW}$, the four A$_2$-configurations on the Enriques surface $\WW$ 
become components of singular fibers of the fibration
of type \eqref{eq-nppss}. 
The latter implies the erroneous claim that  ${\mathcal A}$ never contains four $3$-divisible sets  
(\cite[Lemma~3.5~(2)]{kz}) and
the fundamental group $\pi_1(\WW^0)$ of the open Enriques surface is either $\ZZ/6\ZZ$ or $S_3 \times \ZZ/3\ZZ$ 
(see  \cite[Lemma~3.6~(3)]{kz}).
Here we correct these claims.

%

\begin{lemm} \label{lem-41}
 Let $\WW$ be an Enriques surface with four A$_2$-configurations ${\mathcal A}$. Then
\begin{equation} \label{eq-possiblegroups} 
\pi_1(\WW^0) \in \{S_3 \times \ZZ/3\ZZ,  \ZZ/6\ZZ,  (\ZZ/3\ZZ)^{\oplus 2} \times \ZZ/2\ZZ \} \, .
\end{equation}

Moreover, one has the following characterizations: 

\begin{enumerate}
\item[(a)] Both ${\mathcal A}$ and ${\mathcal A}^{\pm}$  contain exactly one $3$-divisible set iff
$\pi_1(\WW^0) = \ZZ/6\ZZ.$

\item[(b)]
 ${\mathcal A}$ contains exactly one $3$-divisible set and 
${\mathcal A}^{\pm}$ contains four $3$-divisible sets iff
$$
\pi_1(\WW^0) = S_3 \times \ZZ/3\ZZ.
$$

\item[(c)] ${\mathcal A}$ contains  four $3$-divisible sets iff
$
\pi_1(\WW^0) = (\ZZ/3\ZZ)^{\oplus 2} \times \ZZ/2\ZZ.
$
\end{enumerate}
\end{lemm}

\begin{proof}
The proof  follows almost verbatim the first part of the proof of  \cite[Lemma~3.6~(3)]{kz},
but there is one addition to be made:
 Lemma~\ref{lem-f3333bigcode} shows that one cannot use  \cite[Lemma~3.5~(2)]{kz} to rule out 
the existence of an Enriques surface $S$ with 
$\pi_1(\WW^0) = (\ZZ/3\ZZ)^{\oplus 2} \times \ZZ/2\ZZ$.
\end{proof}

With these preparations we can prove the following precise version of Theorem \ref{thm1}:

\begin{thm} \label{thm} 
Let $\WW$ be an Enriques surface with a set of four mutually disjoint A$_2$-configurations ${\mathcal A}$. 
\begin{enumerate}
\item[(a)] If $(\WW, {\mathcal A}) \in {\mathcal F}_{4,3,1}$, then
$
\pi_1(\WW, {\mathcal A}) = S_3 \times \ZZ/3\ZZ$.
\item[(b)] If $\WW \in {\mathcal F}_{3,3,3,3}$, then there exist ${\mathcal A}'$ and ${\mathcal A}''$ such that 
$$
\pi_1(\WW, {\mathcal A}') = (\ZZ/3\ZZ)^{\oplus 2} \times \ZZ/2\ZZ \mbox{ and } \pi_1(\WW, {\mathcal A}'') = \ZZ/6\ZZ.
$$ 
\end{enumerate}
In particular, all groups given in Lemma \ref{lem-41} are realized by Enriques surfaces.
\end{thm}

\begin{proof}
We start with (b) which is much easier to prove.
Indeed, 
the existence of ${\mathcal A}'$ (resp. ${\mathcal A}''$) results immediately  
from Lemma~\ref{lem-f3333bigcode} and Lemma~\ref{lem-41}~(c)
(resp. Lemma~\ref{lem-f3333smallcode}, Remark \ref{rem:I_3} and Lemma~\ref{lem-41}~(a)).

 In order to prove  part (a) we assume that $(\WW, {\mathcal A}) \in {\mathcal F}_{4,3,1}$, i.e. that there exist the fibration $|2 \HE|$ and  the map 
$\PLpsi_{\WW}$. By   Lemma~\ref{lem-41} it suffices to show that $\mathcal A^{\pm}$ contains four 3-divisible sets,
but $\mathcal A$ contains only one 3-divisible set.

We label the components of the type-$IV^*$ fiber of the fibration $|2 \HE|$ as in \eqref{eq-myfiber}. Moreover, 
we put  $\Theta_7, \ldots, \Theta_9$ to denote the components of the $I_3$ fiber. We can assume that \eqref{eq-useful-en},  \eqref{eq-useful-K3}
hold. Then, the divisor 
\begin{equation}
\label{eq:K3-3}
\Theta_1^+ - \Theta_2^+ + \Theta_3^+ - \Theta_4^+ + \Theta_5^+ - \Theta_6^+ - \Theta_1^- + \Theta_2^- - \Theta_3^- +  \Theta_4^- -\Theta_5^- +  \Theta_6^- 
\end{equation}
is 3-divisible (see the proof  of \cite[Lemma~3.5.1]{kz}). From  \eqref{eq-useful-K3} we infer that  $\mathcal A^{\pm}$ contains a 3-divisible set
that cannot be obtained as the image of a 3-divisible set in  $\mathcal A$ under the pull-back $\proj^*$. On the other hand, the eight A$_2$ configurations contain the pull-back $\proj^*$ of a
3-divisible set contained in    $\mathcal A$. The latter exists by Cor.~\ref{eq-2poss}. 
This proves the first claim.

The second claim is a little more subtle due to the multiples of the half-pencil $H$ involved.
Indeed, the 4 A$_2$ configurations have to contain a 3-divisible set by Lemma~\ref{lem-f431}, but
the three A$_2$ configurations $\Theta_1,\hdots,\Theta_6$ supported on the $IV^*$ fiber are themselves \emph{not}
3-divisible because the divisor
\begin{eqnarray}
\label{eq:D_0}
D_0 = \Theta_1-\Theta_2+\Theta_3-\Theta_4+\Theta_5-\Theta_6+H
\end{eqnarray}
visibly is (in the 3-divisible divisor \eqref{eq:K3-3} on the covering K3, the contributions from $H$ even out).
That is, without multiples of the half-pencil, the 4 A$_2$ configurations contain only one 3-divisible set
(by Lemma~\ref{lem-f431})
while modulo the half-pencil, there are four 3-divisible sets.
Therefore we will have to conduct a careful analysis of the precise multiples of $H$ involved.

After rearranging the $(-2)$-curves, if necessary, we can assume that the divisor
\begin{eqnarray}
\label{eq:D_1}
D_1 =  \Theta_1-\Theta_2-\Theta_3+\Theta_4+\Theta_7-\Theta_8
\end{eqnarray}
is 3-divisible.
By symmetry, the same applies to a compatible pair of divisors
\begin{eqnarray*}
D_2 & = &  \Theta_3-\Theta_4-\Theta_5+\Theta_6+
\begin{cases}
\Theta_8-\Theta_9, & \text{ case (i)}\\
\Theta_9 - \Theta_7, &  \text{ case (ii)}
\end{cases}\\
D_3 & = & \Theta_5-\Theta_6-\Theta_1+\Theta_2+
\begin{cases}
\Theta_9 - \Theta_7, &  \text{ case (i)}\\
\Theta_8-\Theta_9, &  \text{ case (ii)}
\end{cases}
\end{eqnarray*}
Which case actually persists depends on the question 
whether the $I_3$ fiber supported on $\Theta_7,\Theta_8,\Theta_9$ is multiple or not,
i.e. whether
\[
\Theta_7+\Theta_8+\Theta_9 = H \;\; \text{ or } \;\; 2H.
\]
One easily verifies that if the fiber is multiple, case (i) persists
while the unramified fiber type leads to case (ii).
It remains to check whether having four 3-divisible sets is compatible with fixed $H$-multiplicities $m_1,\hdots,m_4$ 
attached to the 4 A$_2$-configurations.
To this end, we may assume that
\[
\mf(F_i'-F_i'') = \Theta_{2i-1}-\Theta_{2i}+m_iH, \;\;\; i=1\hdots,4.
\]
Since we only care about 3-divisibility, all equations in the remainder of the proof of Theorem \ref{thm} 
should be understood in $\ZZ/3\ZZ$.
Then \eqref{eq:D_0} leads to 
\begin{eqnarray*}
\label{eq0}
m_1+m_2+m_3 & = & 1
\end{eqnarray*}
while \eqref{eq:D_1} gives
\begin{eqnarray*}
\label{eq1}
m_1-m_2+m_4 & = & 0
\end{eqnarray*}
The two cases for $D_2, D_3$ lead to the same equations
although (or because) they depend on the multiplicity of the $I_3$ fiber:
\begin{eqnarray*}
m_2-m_3+m_4 & = & 1\\
-m_1 + m_3 + m_4  & = & -1
\end{eqnarray*}
Thus we obtain a system of four linear equations over $\ZZ/3\ZZ$.
One immediately verifies that it has no solution.
Hence, regardless of the multiples of $H$ involved,
the four A$_2$ configurations cannot support four 3-divisible sets.
By Lemma \ref{lem-41}, this concludes the proof of Theorem \ref{thm}.
\end{proof}

%


Finally, we use the jacobian fibration to verify that surfaces in ${\mathcal F}_{4,3,1} \cap {\mathcal F}_{3,3,3,3}$ 
have some special properties;
notably the families ${\mathcal F}_{4,3,1}, {\mathcal F}_{3,3,3,3}$ only overlap on proper subfamilies:

\begin{prop} \label{prop-ro18}
Let $\WW$ be an Enriques surface 
 and let $\XXX$ be the K3-cover of $\WW$.
If $\WW \in {\mathcal F}_{4,3,1} \cap {\mathcal F}_{3,3,3,3}$, then  $\rho(\XXX) \geq  19$.
\end{prop}

\begin{proof} 
We compare the discriminants of the K3-covers.
For $\WW\in\Ft$ with K3-cover $\XXX$ of Picard number $\rho(\XXX)=18$,
Lemma~\ref{lem:324} gives $\operatorname{d}(\mbox{NS}(\XXX)) = -324$.

A completely analogous argument applies to $\WW' \in  {\mathcal F}_{4,3,1}$ 
 with K3-cover $\XXX'$ such that  $\rho(\XXX') = 18$. 
We find that $\operatorname{d}(\mbox{NS}(\XXX')) = -36$.
This implies that $\WW' \not \in  {\mathcal F}_{3,3,3,3}$ and vice versa for $\WW$.
\end{proof}

In the next section, we shall take this result as a starting point to take a closer look at
the moduli of our families $\Ft$ and $\Ff$.

\section{Moduli}
\label{s:moduli}

\subsection{Algebro-geometric construction}
\label{ss:6321}

We start by giving an algebro-geometric description of the family $\Ff$.
As opposed to the analytic construction of logarithmic transformations,
it will be based on Enriques involutions of base change type as outlined in \ref{ss:ell}.
Our starting point is another extremal rational elliptic surface $\Xs$,
this time with $\MW(\Xs)\cong \ZZ/6\ZZ$.
As  a cubic pencil, it can be given by 
\begin{eqnarray}
\label{eq:6321}
\Xs:\;\; (x+y)(y+z)(z+x)+t xyz=0.
\end{eqnarray}
More precisely, $\Xs$ is the relatively minimal resolution of the above cubic pencil model in $\PP^2\times\PP^1$,
obtained by blowing up the three double base points at $[1,0,0],[0,1,0],[0,0,1]$.
The blow-up results in a fiber of Kodaira type $I_6$ at $\infty$; the other singular fibers are $I_3$ at $t=0$, $I_2$ at
$t=1$ and $I_1$ at $t=-8$.
The other three base points of the cubic pencil are actually points of inflection.
Fixing one of them as zero $O$ for the group law, say $[1,-1,0]$, we find that
$P=[0,0,1]$ has order $2$ inside $\MW(\Xs)$.
Thus it lends itself to (the classical case of) the construction of an Enriques involution of base change type (as reviewed in \ref{ss:ell}).
To this end,
consider a quadratic base change \eqref{eq:base-change}
that does not ramify at the $I_3$ and $I_1$ fibres.
Denote the pull-back surface by $\Ys$;
this is an elliptic K3 surface, generically with all singular fibres of $\Xs$ duplicated.
The deck transformation $\imath$ enables us to define a fixed point-free involution $\psi$ on $\Ys$ by
\[
\psi =  (\text{fibrewise translation by } P) \circ \imath
\]
The quotient surface will be an Enriques surface $\WW=\Ss$
with elliptic fibration 
\[
\pi_0: \WW \to \PP^1
\]
 with the same singular fibers as $\Xs$ (generically not multiple);
here $O$ and $P$ map to a smooth rational bisection $R$.

\begin{lemm}
$S$ contains 4 perpendicular A$_2$ configurations.
\end{lemm}

\begin{proof}
Consider the singular fibres of $\WW$ together with the bisection $R$.
The following figure depicts how they intersect
and indicates the 4 A$_2$ configurations.
\end{proof}

\begin{figure}[ht!]
\setlength{\unitlength}{.45in}
\begin{picture}(8,2.7)(3,0)
\thicklines


\multiput(4,2)(2,0){2}{\circle*{.1}}
\multiput(4,1)(2,0){2}{\circle*{.1}}
\put(4,2){\line(0,-1){1}}
\put(6,2){\line(0,-1){1}}
\put(5,0.4){\circle*{.1}}
\put(5,2.6){\circle*{.1}}


\put(4,2){\line(5,3){1}}
\put(5,0.4){\line(5,3){1}}
\put(4,1){\line(5,-3){1}}
\put(5,2.6){\line(5,-3){1}}

%

\put(7.5,2){\circle*{.1}}
\qbezier(5,2.6)(6.5,2.6)(7.5,2)
\qbezier(5,.4)(6.5,.4)(7.5,2)


\put(9,2){\circle*{.1}}
\put(7.5,1.95){\line(1,0){1.5}}
\put(7.5,2.05){\line(1,0){1.5}}

\put(10,2.8){\circle*{.1}}
\put(10,1.2){\circle*{.1}}

\put(9,2){\line(5,4){1}}
\put(9,2){\line(5,-4){1}}

\put(10,2.8){\line(0,-1){1.6}}


\put(7.5,.5){\circle*{.1}}
\put(8.5,.5){\circle*{.1}}

\put(7.5,.45){\line(1,0){1}}
\put(7.5,.55){\line(1,0){1}}

\put(7.5,2){\line(0,-1){1.5}}
\put(7.5,2){\line(2,-3){1}}

%
\put(6.9,1.85){$R$}
\put(8.5,0){$I_2$}
\put(9.2,1.2){$I_3$}
\put(4.05,.4){$I_6$}
%
\thinlines
\put(3.8,0.8){\framebox(.4,1.4){}}
\put(5.8,0.8){\framebox(.4,1.4){}}
\put(9.8,1){\framebox(.4,1.9){}}

\put(7.3,0.3){\framebox(.4,1.9){}}
\end{picture}
\caption{4 disjoint A$_2$-configurations on $\Ss$}
\label{Fig:1}
\end{figure}
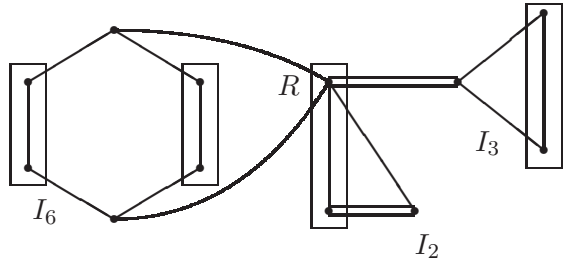

By Theorem \ref{thm0}, we conclude that $S\in\Ft$ or $S\in \Ff$.
For discriminant reasons (compare the proof of Proposition \ref{prop-ro18}), the second alternative should hold.
Here we will give a purely geometric argument:

\begin{lemm}
\label{lem:431}
$S\in\Ff$.
\end{lemm}

\begin{proof}
It suffices to identify a divisor of Kodaira type $IV^*$ on $S$ with orthogonal $A_2$.
Then its linear system will induce an elliptic fibration 
\[
\pi: S\to\PP^1
\] 
with singular fibres of types $IV^*$ and $I_3$ or $IV$;
thus $S\in\Ff$. 
This is easily achieved: simply connect the three A$_2$'s on the left in Figure \ref{Fig:1} through one of the remaining components
of the original $I_6$ fibre.
\end{proof}

\begin{rem}
\label{rem:bi}
A bisection for the fibration $\pi$ from the proof of Lemma \ref{lem:431} 
can be given without difficulty
(although Figure \ref{Fig:1} only displays fibre components and 4-sections): 
take a half-pencil $B$ of the fibration $\pi_0$.
Out of the curves depicted in Figure \ref{Fig:1},
$B$ only meets the bisection $R$ with multiplicity one.
On the fibration $\pi$,
$B$ meets the $IV^*$ fibre only in the double component $R$;
since additive fibers cannot be multiple,
$B$ thus defines a bisection for $\pi$.
The fiber of type $I_3$ or $IV$ is met by $B$ only in the component not displayed 
in Figure \ref{Fig:1}.
\end{rem}

\subsection{K3-cover for $\Ff$}

Overall, there are 6 configurations
how a bisection may intersect the two reducible fibers of a given Enriques surface $\WW\in\Ff$.
For 3 of them, including the one sketched in Remark \ref{rem:bi},
we can conversely derive the configuration of rational curves on $\Ss$ originating from the
Enriques involution of base change type.
Here we detail on one example:

\begin{examp}
\label{exx}
For the configuration from Remark \ref{rem:bi}
comprising a bisection $B$ (of arithmetic genus zero) 
meeting the fibres of type $IV^*$ and $I_3/IV$ on $\WW$,
one finds that $B$ automatically is a half-pencil inducing an elliptic  fibration $\pi_{|2B|}$,
since it is met by some (nodal) curve (a double component $\Theta$ of the $IV^*$ fibre) 
with multiplicity $1$ (so $\Theta$ gives a bisection for $\pi_{|2B|}$).
This fibration has singular fibers accounting for the root lattices $A_5+A_1$ and $A_2$
obtained from the extended Dynkin diagrams $\tilde E_6, \tilde A_2$ by omitting the curves meeting $B$.
With the nodal bisection $\Theta$, we necessarily end up on a  quotient of $\Xs$
by an Enriques involution of base change type.
Indeed, otherwise, there would be an additive fiber  of type $IV^*, III^*$ or $II^*$;
but here it is easy to see that the root of the $A_5$-diagram met by $\Theta$ 
would correspond to a triple fiber component which, of course, cannot be met by a bisection.
\end{examp}

For the other three possible configurations of bisection and singular fibers,
the approach from Example \ref{exx} does not seem to work.
However, the next proposition and its corollary show
in a lattice-theoretic, hence non-explicit way that also these Enriques surfaces are 
 covered by $\Ys$.
We expect that they 
arise from $\Ys$
by another kind of Enriques involution.

\begin{prop}
\label{prop:NS}
Let $\WW\in\Ff$ such that the K3 cover $\XXX$ has $\rho(\XXX)=18$.
Then 
\[
\NS(\XXX) \cong U(2) + A_2 + E_6 + E_8.
\]
\end{prop}

\begin{proof}
The elliptic fibration on
$\XXX$ induced from $\WW$ comes automatically with a bisection $R$.
Since $\rho(\XXX)=18$, we can assume that $R^2=0$.
The key step in proving the proposition is the observation
that we can modify $R$ to a divisor $D$ by adding fiber components as correction terms
such that $D$ is perpendicular to 2 A$_2$ and 2 E$_6$ configurations on $\XXX$
(in the fibers of $\pi$).
For fibers of type $I_3, IV$, this has been exhibited in the proof of Lemma \ref{lem-f3333bigcode}.
For $IV^*$ fibers, it is a similar exercise.
For instance, if $R$ meets a double component, then simply subtract from $R$ 
the adjacent simple component $\Theta'$
(which is thus met by $R-\Theta'$ with multiplicity 2).

Crucially, we now use that the singular fibers come in pairs 
which are met by $R$ in exactly the same way
(there cannot be non-reduced singular fibers since $\rho(\XXX)=18$).
In consequence, the correction terms for $D$ also come in pairs,
so
\[
D^2 \equiv R^2 \equiv 0\mod 4.
\]
Hence $D$ and the general fiber $F$ span the lattice $U(2)$,
and we obtain a finite index sublattice
\[
U(2) + A_2^2 + E_6^2 \subset \NS(\XXX).
\]
To compute $\NS(\XXX)$,
it remains to take the 3-divisible class in $A_2^2 + E_6^2$
induced from the Enriques surface into account.
From the lattice viewpoint, this behaves exactly like the 3-torsion section on $\Jac(\XXX)$;
we obtain an integral index 3 overlattice $M$ of $A_2^2 + E_6^2$
by adjoining a vector of square $-4$ obtained by adding up minimal vectors twice each 
of $A_2^\vee$
and $E_6^\vee$ (of square $-2/3$ resp. $-4/3$).
Finally we verify that $M$ and $A_2+E_6$ have isomorphic discriminant forms.
By \cite[Cor. 1.13.3]{nikulin}, it follows that $U(2)+M$
and $U(2)+A_2+E_6+E_8$ are isometric.
\end{proof}

\begin{cor}
A general Enriques surface $S\in\Ff$ is covered by a K3 surface $\Ys$.
\end{cor}

\begin{proof}
The N\'eron-Severi lattice of the covering K3 surface
admits a unique embedding into the K3 lattice $U^3+E_8^2$
up to isometries by \cite[Thm. 1.14.4]{nikulin}.
Hence
the K3 surfaces with this lattice polarisation form an irreducible two-dimensional family,
and the corollary ends up being a consequence of Lemma \ref{lem:431} in the reverse direction
(since the quadratic base changes of $\Xs$ exactly form a two-dimensional family,
the parameters being the non-ordered pairs comprising the two ramification points).
\end{proof}

\begin{rem}
\label{rem:j=0}
We can use the above description to give a proof of Lemma \ref{lem:boundary}.
For this purpose, we check within our family where the singular fiber types degenerate from $(I_3 + I_1)$ to $IV$.
For this, we normalise the base change \eqref{eq:base-change}
to take the shape
\[
t \mapsto 1-\gamma \dfrac{(t-1)(t-\lambda)}t,
\]
so that $\Ys$ has fibers of type $I_6$ at $0, \infty$
and $I_2$ at $1, \lambda$.
Then we extract the elliptic fibration with 2 fibers of types $IV^*$ and 2 perpendicular $A_2$
inducing generically a quadratic base change of $\Xf$.
This turns out to be isotrivial (with zero j-invariant and $IV$ fibers instead of $I_3$ and $I_1$) exactly for\linebreak
 $\lambda=1-3/\gamma$.
\end{rem}

\subsection{Comparison with $\Ys$}
\label{ss:sanity}

As a sanity check,
we will compute $\NS(\Ys)$ at a very general moduli point directly.
Incidentally, this will allow us to draw interesting consequences,
see Theorem \ref{thm:Hodge}.

Consider a K3 surface $\Ys$ with $\rho(\Ys)=18$.
By \cite[(22)]{shioda-schuett}, $\NS(\Ys)$ has discriminant $-36$.
In order to compute $\NS(\Ys)$ directly,
we will identify  two perpendicular divisors $D_1, D_2$ of Kodaira type $II^*$ among the plentitude of $(-2)$-curves
visible in the elliptic fibration $\pi_0$ as fiber components and torsion sections.
To define $D_1$,
connect the zero section $O$ in three directions:
by a component of either $I_2$ fiber, two components of an $I_3$ and a chain $\Theta_0,\hdots\Theta_4$ 
of five components of an $I_6$.
Similarly, the divisor $D_2$ comprises the $6$-torsion section disjoint from $D_1$ 
(i.e. meeting the remaining fibre component $\Theta_5$ of the chosen $I_6$ fiber)
and fibre components of the other $I_2, I_3$ and $I_6$ fibres.

This approach has several advantages.
First it reveals that $\Ys$ admits an elliptic fibration $\pi_{|D_1|}$ with two fibers of type $II^*$.
This comes with  multisections of degree $6$, given for instance by  $\Theta_5$. 
In consequence, the Jacobian has 
\[
\NS(\Jac(\Ys,\pi_{|D_1|})) \cong U+E_8^2.
\]
With this N\'eron-Severi lattice, $\Jac(\Ys,\pi_{|D_1|})$ is sandwiched by the Kummer surface of two elliptic curves
by \cite{Sandwich}.
All of this occurs in the framework of Shioda-Inose structures and shows the following result:

\begin{thm}
\label{thm:Hodge}
There are elliptic curves $E, E'$ such that as transcendental $\QQ$-Hodge structures
\[
T(\Ys) \cong T(E\times E').
\]
\end{thm}

\begin{rem}
Theorem \ref{thm:Hodge} provides a conceptual way to exhibit explicit K3 surfaces $\Ys$ with $\rho(\Ys)=18$,
parallelling \cite[\S 4.7]{GS}.
Using the involution of base change type from \ref{ss:6321},
we obtain explicit very general members of the family $\Ff$,
as opposed to the extraordinary Example \ref{example-kondo}. 
\end{rem}

As a second application, we return to the computation of $\NS(\Ys)$.
Consider the orthogonal projection inside $\NS(\Ys)$ with respect to the sublattice $E_8^2$ 
specified above.
The multisection $\Theta_5$ is taken to a divisor $D$ of square $D'^2=-60$
orthogonal to $E_8^2$.
It follows that $D$ and a fiber of $\pi_{|D_1|}$ generate the lattice $U(6)$.
Thus we obtain 
\begin{eqnarray}
\label{eq:U(6)}
\NS(\Ys) \cong U(6) + E_8^2.
\end{eqnarray}
(Here, a priori we only have checked the inclusion '$\supseteq$',
but equality holds since the discriminants match.) 
One easily checks that the discriminant forms of the N\'eron-Severi lattices in Proposition \ref{prop:NS} and in
\eqref{eq:U(6)} agree.
By \cite[Cor. 1.13.3]{nikulin}, this suffices to prove that the lattices are isometric as required.

\begin{prop}
\label{prop:T}
If $\Ys$ has $\rho(\Ys)=18$, then $T(\Ys)\cong U+U(6)$.
\end{prop}

\begin{proof}
This follows directly from \eqref{eq:U(6)} using \cite[Prop. 1.6.1 \& Cor. 1.13.3]{nikulin}. 
\end{proof}

\subsection{K3-cover for $\Ft$}
\label{ss:K3-3}

We can carry out similar calculations for the K3 cover $\XXX'$
of an Enriques surface $\WW'\in\Ft$.
Here we only sketch the results.

Generically, $\XXX'$ comes equipped with an elliptic fibration with 8 fibers of type $I_3$
and an irreducible bisection $R'$ such that $R'^2=0$.
Thus the argumentation from the proof of  Lemma \ref{lem-f3333bigcode} applies to modify $R'$ to a divisor $D$ perpendicular to 8 disjoint A$_2$ configurations (supported on the fibres).
Generically, we obtain the finite index sublattice
\[
U(2) + A_2^8\hookrightarrow \NS(\XXX')
\]
which leads to the following analogue of Proposition \ref{prop:NS}

\begin{prop}
\label{prop:NS'}
Let $\WW'\in\Ft$ such that the K3 cover $\XXX'$ has $\rho(\XXX')=18$.
Then 
\[
\NS(\XXX') \cong U(2) + A_2^2 + E_6^2.
\]
\end{prop}

As before, it follows that the K3 covers of all Enriques surfaces in $\Ft$
form an irreducible two-dimensional family.
Using the discriminant form, we can compute the transcendental lattice of a very general K3 cover:

\begin{prop}
\label{prop:T'}
Let $\WW'\in\Ft$ be an Enriques surfaces such that its K3-cover $\XXX'$ has $\rho(\XXX')=18$.
Then 
\[
T(\XXX')\cong U(3)+U(6).
\]
\end{prop}

\begin{proof}
The discriminant group $A_{\NS}$ of $\NS(\XXX')$ has $3$-length $4$ by Proposition \ref{prop:NS'}.
Since this length equals the rank of $T(\XXX')$,
we deduce that $T(\XXX')$ is $3$-divisible as an integral even lattice,
i.e.
\[
T(\XXX') = M(3) \;\;\; \text{ for some even lattice } M.
\]
By Lemma \ref{lem:324}, $M$ has determinant $4$.
From Proposition \ref{prop:NS'}, we infer the equality of discriminant forms
\[
q_M = - q_{U(2)} = q_{U(2)}.
\]
Hence $M\cong U+U(2)$ by \cite[Prop. 1.6.1 \& Cor. 1.13.3]{nikulin} 
\end{proof}

\subsection{Overlap of $\Ff$ and $\Ft$}

Recall from Proposition \ref{prop-ro18}
that the two families of Enriques surfaces $\Ff$ and $\Ft$ only intersect
on one-dimensional subfamilies.
Here we shall give a lattice theoretic characterisation
of two infinite series of subfamilies
and work out the first case explicitly.

In essence, computing the one-dimensional subfamilies of overlap  amounts
to calculating even lattices $T$ of signature $(2,1)$
admitting primitive embeddings into both generic transcendental lattices 
from Propositions \ref{prop:T}, \ref{prop:T'}.
Then one can enhance the N\'eron-Severi lattices by a primitive vector perpendicular to $T$
using the gluing data encoded in the discriminant form (see \cite[\S 3]{GS}, e.g.).
There are two obvious kinds of candidates for $T$ with $N\in\ZZ_{>0}$:

\begin{eqnarray}
\label{eq:12}
U(3) + \langle 12N\rangle & \hookrightarrow & \begin{cases}
U(3) + U(2) \cong U + U(6) \\
U(3) + U(6)
\end{cases}\\
U(6) + \langle 6N\rangle & \hookrightarrow & \begin{cases}
 U(6) + U\\
U(6) + U(3)
\end{cases}
\end{eqnarray}

\begin{rem}
\label{rem:jac}
We point out that
\eqref{eq:12}
includes families where the Jacobians of the K3 covers $\Yt$ and $\Yf$ overlap.
In fact, this happens with transcendental lattices $U(3) + \langle 6M\rangle$, 
and one can show as in \cite[Prop. 4.2]{hm1}
that a K3 surface with this transcendental lattice admits an Enriques involution
if and only if $M$ is even.
Moreover, the involution turns out to be of base change type,
so we can, at least in principle, give a very explicit description of these surfaces.
\end{rem}

\subsection{Explicit component of $\Ft\cap\Ff$}
\label{ss:explicit}

We conclude this paper by working out the first case of \eqref{eq:12} explicitly.
That is, we aim for K3 surfaces with transcendental lattice
\begin{eqnarray}
\label{eq:T-ex}
T = U(3) + \langle 12\rangle.
\end{eqnarray}
By Remark \ref{rem:jac},
this could be done purely on the level of Jacobians of $\Yt$ or $\Yf$,
but here we shall rather continue to work with $\Ys$.
The lattice enhancement raises the rank of the N\'eron-Severi lattice by one
while the discriminant changes from $-36$ to $108$.
By the theory of Mordell-Weil lattices \cite{ShMW}, this can only be achieved by adding
a section $Q$ of height $3$.
Up to adding a torsion section,
we may assume $Q$ to be induced by the quadratic twist $\RES'$ of $\Xs$
corresponding to the quadratic base change \eqref{eq:base-change}.
I.e. $Q$ comes from a section $Q'$ of height $3/2$ on $X'$.
Note that $X'$ inherits the 2-torsion section from $\Xs$.
At the same time, this will ease the explicit computations and
limit the possible configurations for $Q'$.
Indeed, using the height formula from \cite{ShMW} 
it is easy to see that there are only two possible cases for $Q'$ 
up to adding the two-torsion section:
\begin{itemize}
\item
either $Q'$ meets exactly one $I_0^*$ fiber (at a component not met by the 2-torsion section) and the $I_6$ fiber  (at the component met by the 2-torsion section) non-trivially,
\item
or it intersects non-trivially 
exactly one $I_0^*$ fiber (at a component not met by the 2-torsion section),
the $I_6$ fibre (at a component adjacent to the zero component),
and the $I_3$ fiber.
\end{itemize}
We can compute the N\'eron-Severi lattice and
the transcendental lattice of the resulting covering K3 surfaces
by the same means as in \ref{ss:sanity}:
simply compute the rank 3 orthogonal complement of $E_8^2$ inside $\NS$.
We obtain $T=U+\langle 108\rangle$ for the second case and 
the desired transcendental lattice from \eqref{eq:T-ex} for the first.

We continue to work out the first case in more detail.
Let us assume that the quadratic base change \eqref{eq:base-change}
ramifies at $a,b\in\PP^1$.
For ease of computations, we shall use an extended Weierstrass form of $X'$
which locates the 2-torsion section at $(0,0)$:
\[
X':\;\;\; y^2 = x(x^2 +(t-a)(t-b) (t^2/4+t-2)x+(t-a)^2(t-b)^2(1-t))
\]
Then we can implement the section $Q'$ to have the $x$-coordinate $c(t-a)$.
Solving for this to give a square upon substituting into the extended Weierstrass form leads to
\[
a = -\frac 13\dfrac{(b+2)^2}{b-4}, \;\;\;  c =(b+8)(b-1)^2/27.
\] 
Thus we obtain explicitly a one-dimensional family of K3 surfaces 
with transcendental lattice \eqref{eq:T-ex}.
Unless the base change degenerates or the ramification points hit the fibers of type $I_1$ or $I_3$,
i.e. for
$
b \not\in\{ -8,-2,0,1,10\}
$,
the resulting K3 surface $\XXX$ possesses the Enriques involution $\psi$
of base change type constructed in \ref{ss:6321}.
By Lemma \ref{lem:431}, the quotient surface $S$ lies in $\Ff$. 
We can also verify geometrically that $S\in\Ft$.
To this end
we use that the induced section $Q$ of height $3$ on $\XXX$
meets only the two $I_6$ fibers non-trivially -- in the same component as the 2-torsion section $P$,
i.e. opposite the zero component --
and it meets the zero section $O$ in the ramified fiber above $t=b$.
Hence $Q, O$ and the identity component of either of the $I_2$ fibers form a triangle,
i.e. they give a divisor $D$ of Kodaira type $I_3$.
Perpendicular to $D$, we find 
\begin{itemize}
\item
another $I_3$ formed by the sections $P, (P-Q)$ and the non-identity component of the other $I_2$ fiber;
\item
6 A$_2$'s contributed from the $I_6$ and $I_3$ fibres;
\item
4 sections of the induced elliptic fibration $\pi_{|D|}$ given by the remaining components of the $I_6$ fibers.
\end{itemize}
We conclude that $\pi_{|D|}$ is a jacobian elliptic fibration with 8 fibers of type $I_3$.
Hence it comes from $\Xt$ by some quadratic base change.
Finally, one directly verifies that the above rational curves on $\XXX$ are interchanged by 
the Enriques involution $\psi$.
Therefore, $\pi_{|D|}$ induces an elliptic fibration with 4 fibers of type $I_3$
on $S=\XXX/\psi$.
That is, $S\in\Ft$ as claimed.

\subsection*{Acknowledgements} 
We are grateful to the  referees for their crucial comments.
We would like to thank M. Joumaah,
R.~Kloosterman, H.~Ohashi and C.~Peters
for helpful  discussions,
and especially to JH Keum for pointing out the special case of the $II^*$ fibration
in the proof of Lemma \ref{lem-1-2}.
We started talking about 
this project  in March 2011 when Sch\"utt enjoyed the hospitality of  Jagiellonian University in Krakow.
Special thanks to S\l awomir Cynk.

\end{document}